\documentclass[a4paper,final]{amsart}
\usepackage[margin=0.8in]{geometry}

\usepackage{amsmath}%
\usepackage{amsfonts}%
\usepackage{amssymb}%
\usepackage{amsthm}
\usepackage{graphicx}
\usepackage{dsfont}
\usepackage{pdfcomment}
\usepackage{tikz}
\usepackage{tikz-3dplot}
\usepackage{subfigure}
\usepackage[colon,sort,square,numbers]{natbib}

\usepackage{thmtools}
\usepackage{thm-restate}
\usepackage{hyperref}
\usepackage{cleveref}
\usepackage{esvect}

\usepackage[parfill]{parskip}
\begingroup
\makeatletter
\@for\theoremstyle:=definition,remark,plain\do{%
	\expandafter\g@addto@macro\csname th@\theoremstyle\endcsname{%
		\addtolength\thm@preskip\parskip
	}%
}
\endgroup

\usetikzlibrary{decorations.markings}
\newtheorem{theorem}{Theorem}[subsection]

\newtheorem{corollary}[theorem]{Corollary}

\newtheorem{definition}[theorem]{Definition}
\newtheorem{example}{Example}
\allowdisplaybreaks

\newtheorem{lemma}[theorem]{Lemma}

\newtheorem{proposition}{Proposition}
\newtheorem{remark}[theorem]{Remark}

\newcommand{\field}[1]{\mathbb{#1}}

\DeclareMathOperator{\Area}{Area}

\DeclareMathOperator{\Real}{Re}
\DeclareMathOperator{\Imaginary}{Im}
\DeclareMathOperator{\grad}{grad}

\DeclareMathOperator{\area}{Area}
\DeclareMathOperator{\sign}{sign}
\renewcommand{\Re}{\Real}

\renewcommand{\Im}{\Imaginary}

\title[Discrete minimal surfaces]{Discrete minimal surfaces: critical points of the area functional from integrable systems}
\author{Wai Yeung Lam}

\address{Wai Yeung Lam\\
	Technische Universit\"at Berlin\\Institut f\"ur Mathematik\\
	Stra{\ss}e des 17.\ Juni 136\\
	10623 Berlin\\ Germany}

\email{lam@math.brown.edu}

\begin{document}
	
\begin{abstract}
	We obtain a unified theory of discrete minimal surfaces based on discrete holomorphic quadratic differentials via a Weierstrass representation. Our discrete holomorphic quadratic differential are invariant under M\"{o}bius transformations. They can be obtained from discrete harmonic functions in the sense of the cotangent Laplacian and Schramm's orthogonal circle patterns.
	
	We show that the corresponding discrete minimal surfaces unify the earlier notions of discrete minimal surfaces: circular minimal surfaces via the integrable systems approach and conical minimal surfaces via the curvature approach. In fact they form conjugate pairs of minimal surfaces.  
	
	Furthermore, discrete holomorphic quadratic differentials obtained from discrete integrable systems yield discrete minimal surfaces which are critical points of the total area.
\end{abstract}

\thanks{This research was supported by the DFG Collaborative Research Centre SFB/TRR 109 \emph{Discretization in Geometry and Dynamics}. The author would like to thank Ulrich Pinkall and Masashi Yasumoto for comments on the manuscript.}

\date{\today}

\maketitle

\section{Introduction}\label{sec:intro}

Minimal surfaces in Euclidean space are classical in Differential Geometry. They arise in the calculus of variations, in complex analysis and are related to integrable systems. A minimal surface is a surface that locally minimizes its area, or equivalently, its mean curvature vanishes identically. The Weierstrass representation for minimal surfaces asserts that locally each minimal surface is given by a pair of holomorphic functions. New minimal surfaces can be obtained from a given minimal surface via Bonnet, Goursat and Darboux transforms.

Structure-preserving discretization is ubiquitous, particularly in Discrete Differential Geometry. The goal of Discrete Differential Geometry is to develop a discrete theory as rich as its smooth counterpart. A common approach is to build upon a discrete analogue of some characterization from the smooth theory. However, several equivalent characterizations in the smooth theory might lead to different discrete theories even if their continuum limits are the same.

In a variational approach, functions are usually defined at vertices while critical points of functionals are sought via vertex-based variations. Pinkall and Polthier \cite{Pinkall1993} considered the total area of a triangulated surface in Euclidean space and suggested to define minimal surfaces as the critical points of the total area. In order to discuss conjugate minimal surfaces and their associated families, the notion of \emph{non-conforming meshes} was introduced  \cite{Polthier2002a}. However it is still not clear how new minimal surfaces can be obtained via Goursat and Darboux transforms.

On the other hand, many surfaces of interest arise with integrable structures, such as constant mean curvature surfaces. Bobenko and Pinkall \cite{Bobenko1996} considered a discrete Lax representation of isothermic surfaces and introduced circular minimal surfaces together with a Weierstrass representation. New minimal surfaces can be obtained via Bonnet, Goursat and Darboux transforms \cite{Hertrich1999}. However, this notion is believed to lack the variational property of minimal surfaces \cite{Polthier2002}.

Other characterizations of surfaces depend on notions of curvature. Smooth minimal surfaces in Euclidean space are characterized by vanishing mean curvature. Curvatures of circular quadrilateral surfaces based on Steiner's formula were proposed by Schief \cite{Schief2004,Schief2007}, compatible with circular minimal surfaces from the integrable systems approach. Bobenko, Pottmann and Wallner \cite{Bobenko2010a} in a similar way defined mean curvature for conical surfaces, which are polyhedral surfaces with face offsets. Conical surfaces with vanishing mean curvature are called conical minimal surfaces. Though special conical minimal surfaces are recently studied \cite{Sechelmann2016}, it is yet unknown if conical minimal surfaces in general admit conjugate minimal surfaces and an analogue of the Weierstrass representation. Furthermore, their relation to the variational approach is not clear.

In this paper we show that the theories of discrete minimal surfaces based on the variational approach, the integrable systems approach and the curvature approach are \emph{not} disjoint. Indeed they possess interesting relations with each other. 

First, we establish a new relation between the integrable systems approach and the curvature approach to discrete minimal surfaces in contrast to \cite{Schief2004,Schief2007}. We define two types of discrete minimal surfaces whose cell decompositions are arbitrary. One of the two types is based on the Christoffel duality of isothermic surfaces. Another is defined via vanishing mean curvature. These two types respectively generalize circular minimal surfaces (from the integrable systems approach) and conical minimal surfaces (from the curvature approach). We show that in our setting each discrete minimal surface of one type corresponds to a discrete surface of the other type, and they form a conjugate pair of minimal surfaces. These surfaces admit a Weierstrass representation in terms of discrete holomorphic quadratic differentials. In particular, each of them corresponds to a discrete harmonic function on a planar triangulated surface as considered in discrete complex analysis \cite{Smirnov2010}.

Second, we disprove the common belief that discrete minimal surfaces via discrete integrable systems lack the variational property. We show that discrete holomorphic quadratic differentials obtained from discrete integrable systems yield discrete minimal surfaces which are critical points of the total area. Every such discrete minimal surface is constructed from a \emph{P-net} \cite{Bobenko1999}, which is half the vertices of a quadrilateral isothermic surface \cite{Bobenko1996} with cross ratios -1. Particular examples of P-nets are given by Schramm's orthogonal circle patterns \cite{Schramm1997}. Bobenko, Hoffmann and Springborn \cite{Bobenko2006} obtained a variational construction of orthogonal circle patterns and showed that each circle pattern corresponds to a s-isothermic minimal surface, which is determined by the combinatorics of the curvature lines of a smooth surface. These discrete minimal surfaces were shown to converge to the smooth ones \cite{Bobenko2006, Matthes2005, Lan2010}. In this paper, we show that by throwing away half the vertices of a s-isothermic minimal surface, the resulting surface not only satisfies our notion of discrete minimal surfaces but is also a critical point of the area functional.

This paper makes use of a generalization of Christoffel duality \cite{Lam2015} and the observation that mean curvature for conical surfaces can be defined without referring to face offsets \cite{Karpenkov2014}. We further define the total area of a discrete surface with non-planar faces via the vector area on faces. As a result of these notions, we obtain connections between the integrable systems approach, the curvature approach and the variational approach to discrete minimal surfaces.

In section \ref{sec:discreteholo}, we introduce discrete holomorphic quadratic differentials with examples to motivate our notions of discrete minimal surfaces.

In section \ref{sec:dismin}, two types of discrete minimal surfaces are defined and shown to be conjugate to each other. Each discrete minimal surface induces an associated family of discrete surfaces with vanishing mean curvature.

In section \ref{sec:Weierstrass}, Goursat transforms of discrete minimal surfaces are constructed in terms of the Weierstrass representation. 

In section \ref{sec:selfstress}, we relate discrete minimal surfaces to self-stresses in the context of the rigidity theory of frameworks.

In section \ref{sec:criticalarea}, we introduce the area of a non-planar polygon and show that discrete surfaces in the associated family of discrete minimal surfaces from a P-net are critical points of the total area.

\section{Notations} \label{sec:notations}

\begin{definition}
	A \emph{discrete surface} is a cell decomposition of a surface $M=(V,E,F)$, with or without boundary. The set of
	vertices (0-cells), edges (1-cells) and faces (2-cells) are denoted by $V$, $E$ and $F$. Furthermore, we write $E_{int}$ as the set of interior edges and $V_{int}$ as the set of interior vertices. Without further notice we assume that all surfaces under consideration are oriented.
\end{definition}

Given a discrete surface $M$, we denote by $M^*=(V^*,E^*,F^*)$ the dual cell decomposition. Each vertex  $i \in V$ corresponds to a dual face $i^* \in F^*$. In particular, interior vertices of $M$ correspond to the interior faces of $M^*$, denoted by $F^*_{int}$, while the boundary vertices of $M$ correspond to the boundary (unbounded) faces of $M^*$.

\begin{definition}
	A \emph{realization} $N:V \to \field{R}^n$ of a discrete surface is an assignment of vertex positions in space. It is \emph{non-degenerate} if $N_i \neq N_j$ for every edge $\{ij\}$. Moreover we say a realization $N:V\to \field{S}^2$ is \emph{admissible} if $N_i \neq -N_j$ for every edge $\{ij\}$. 
\end{definition}

We make use of discrete differential forms from Discrete Exterior Calculus \cite{Desbrun2006a}. Given a discrete surface $M=(V,E,F)$, we denote by $\vec{E}$ the set of oriented edges and by $\vec{E}_{int}$ the set of interior oriented edges. An oriented edge from vertex $i$ to vertex $j$ is indicated by $e_{ij}$. A function $\omega:\vec{E}\to \field{R}$ is called a (primal) \emph{discrete 1-form} if
\[
\omega(e_{ij}) = -\omega(e_{ji}) \quad \forall e_{ij} \in \vec{E}.
\]
It is \emph{closed} if for every face $\phi=(v_0,v_1,\dots,v_n=v_0)$	
\[
\sum_{i=0}^{n-1} \omega(e_{i,i+1})=0.
\]
It is \emph{exact} if there exists a function $h:V \to \field{R}$ such that for $e_{ij} \in \vec{E}$
\[
\omega(e_{ij}) =  h_j -h_i =: dh(e_{ij}).
\]
Similarly, we consider discrete 1-forms on the dual mesh $M^*$ and these are called dual 1-forms on $M$. For every oriented edge $e$, we write $e^*$ as its dual edge oriented from the right face of $e$ to its left face. A function $\eta:\vec{E}^*_{int} \to \field{R}$ defined on oriented dual edges is called a dual 1-form if
\[
\eta(e^*_{ij}) = -\eta(e^*_{ji}) \quad \forall e^*_{ij} \in \vec{E}^*_{int}.
\]
A dual 1-form $\eta$ is \emph{closed} if for ever interior vertex $i \in V_{int}$ 
\[
\sum_j \eta(e^*_{i\!j}) = 0.
\] 
It is exact if there exists $f:F \to \field{R}$ (i.e. $f:V^* \to \field{R}$) such that
\[
df(e^*_{i\!j}):= f_{\phi_l}-f_{\phi_r} = \eta(e^*_{i\!j})
\] 
where $\phi_l$ denotes the left face of $e_{i\!j}$ and $\phi_r$ denotes the right face. In particular, exactness implies closedness:
\[
\sum_j df(e^*_{i\!j}) = 0
\]
for every interior vertex $i\in V_{int}$.

\section{Discrete holomorphic quadratic differentials}\label{sec:discreteholo}

We motivate the definitions of discrete minimal surfaces via a Weierstrass representation in terms of discrete holomorphic quadratic differentials.

In the smooth theory, a surface in Euclidean space is a minimal surface if it locally minimizes its area, or equivalently its mean curvature vanishes identically. The Weierstrass representation for smooth minimal surfaces in $\mathbb{R}^3$ is a classical application of complex analysis: 
\begin{proposition}
	Given a holomorphic function $h: U \subset \mathbb{C} \to \mathbb{C}$ and a meromorphic function $g: U \to \mathbb{C}$ such that $h g^2 $ is holomorphic, the surface $f:U \to \mathbb{R}^3$ defined by
	\begin{equation*}\label{eq:weierstrass}
		df = \Re \left( \left( \begin{array}{c}
			1-g^2 \\ i(1+g^2) \\ 2g 
		\end{array}\right) h(z) dz \right) = \Re \left( \left( \begin{array}{c}
		1-g^2 \\ i(1+g^2) \\ 2g 
	\end{array}\right) \frac{q}{dg} \right)
\end{equation*}
is a minimal surface where $q:= hg_z dz^2$ is a holomorphic quadratic differential. Its Gau{\ss} map $N$ is the stereographic projection of $g$
\begin{equation*}
	N=\frac{1}{|g|^2+1}\left(\begin{array}{c}2\Re g \\ 2 \Im g\\ |g|^2-1 \end{array}\right)
\end{equation*}
and its Hopf differential is $q$, which encodes the second fundamental form: The direction defined by a nonzero tangent vector $W$ is
\begin{equation*}
	\begin{array}{rcl}
		\text{an asymptotic direction} & \iff & q(W) \in i \mathbb{R}, \\
		\text{a principal curvature direction} & \iff & q(W) \in  \mathbb{R}.
	\end{array}	
\end{equation*}		
Locally, every minimal surface can be written in this form.
\end{proposition}

A M\"{o}bius invariant notion of discrete holomorphic quadratic differentials was introduced by \citet{Lam2015a}:

\begin{definition}
	Given a non-degenerate realization $z:V \to \field{R}^2 \cong \field{C}$ of a discrete surface $M=(V,E,F)$ in the plane, a function $q:E_{int} \to \field{R}$ defined on interior edges is a \emph{discrete holomorphic quadratic differential} with respect to $z$ if for every interior vertex $i \in V_{int}$
	\begin{gather*}
		\sum_{j} q_{ij} = 0, \\
		\sum_{j} q_{ij}/ (z_j -z_i) = 0.
	\end{gather*} 
\end{definition}

\begin{proposition}[\cite{Lam2015a}] \label{thm:mobius}
	Suppose $z:V \to \field{C}$ is a non-degenerate realization of a discrete surface and $\Phi:\field{C}\cup \{\infty\} \to \field{C}\cup \{\infty\}$ is a M\"{o}bius transformation which does not map any vertex to infinity. Then $q$ is a holomorphic quadratic differential on $z$ if and only if $q$ is a holomorphic quadratic differential on $w:= \Phi \circ z$.
\end{proposition}

\begin{proof}
	Since M\"{o}bius transformations are generated by Euclidean transformations and inversions, it suffices to consider the inversion in the unit circle at the origin
	\[
	w:=\Phi(z) = 1/z.
	\]
	We have
	\begin{align*}
		\sum_{j} q_{i\!j}/ (w_j-w_i) = \sum_{j} -z_i z_j q_{i\!j}/ (z_j -z_i) = -z_i \sum_{j} q_{i\!j} - z_i^2 \sum_{j} q_{i\!j} /(z_j-z_i).
	\end{align*}
	Thus the claim follows.
\end{proof}

Discrete holomorphic quadratic differentials are closely related to discrete conformality of triangulated surfaces, namely the theory of length cross ratios and circle patterns. It was shown in \cite{Lam2015a} that discrete holomorphic quadratic differentials parametrize the change of the logarithmic cross ratios under infinitesimal conformal deformations.

We first consider examples from discrete harmonic functions in linear discrete complex analysis. Discrete harmonic functions were introduced on the square lattice by Ferrand \cite{Ferrand1944}, McNeal \cite{McNeal1946} and Duffin \cite{Duffin1956} by means of a discrete Cauchy-Riemann equation. This notion was later generalized to triangulated surfaces and led to the cotangent Laplacian. Discrete harmonic functions appear in various contexts, such as the finite-element approximation of the Dirichlet energy \cite{Pinkall1993}, and have applications in statistical mechanics \cite{Smirnov2010}.

\begin{example}[Discrete harmonic functions] Given a non-degenerate realization $z:V \to \field{C}$ of a triangulated disk, a function $u: V \to \field{R}$ is a \emph{discrete harmonic function} in the sense of the cotangent Laplacian if 
	\begin{equation*} \sum_j (\cot \angle jki + \cot \angle ilj) (u_j - u_i) =0  \quad \forall i \in V_{int} \end{equation*}
	where $\{ijk\}$ and $\{jil\}$ are two neighboring faces containing the edge $\{ij\}$. 
	
	\citet{Lam2015a} showed that the space of holomorphic quadratic differentials is a vector space isomorphic to the space of discrete harmonic functions modulo linear functions. A function $u:V \to \field{R}$ defined at the vertices of a triangulated surface can be extended linearly on each face. Its gradient $\grad u:F \to \field{C}$ is constant over faces:
	\[
	(\grad_z u)_{ijk} =  i \frac{u_i dz(e_{jk})+ u_j dz(e_{ki})+u_k dz(e_{ij})}{2 A_{ijk}}. 
	\]
	Then it was proved that every holomorphic quadratic differential $q: E_{int} \to \field{R}$ is of the form
	\begin{align} \label{eq:hessian}
		q_{ij} = idu_z(e^*_{ij}) dz(e_{ij})  \quad  \forall e_{ij} \in \vec{E}_{int}
	\end{align}
	for some harmonic function $u:V \to \field{R}$ where
	\[
	du_z(e^*_{ij}) := \overline{(\grad_z u)_{ijk}} - \overline{(\grad_z u)_{jil}}.
	\]
	It is interesting that Equation \eqref{eq:hessian} also appeared in  \cite[Definition 4]{Polthier2003}. 
\end{example}

On the other hand, discrete holomorphic quadratic differentials are closely related to discrete integrable systems. They appear in the discrete equations of Toda type \cite{Adler2000,Bobenko2003}.

In the following, we focus on examples from discrete holomorphic maps in nonlinear discrete complex analysis, which are known to possess discrete integrable structures. In Section \ref{sec:criticalarea} we show that their corresponding discrete minimal surfaces are critical points of the total area.

\begin{definition} \label{def:pnet}
	A cell decomposition $D=(V,E,F)$ of a simply connected surface is a \emph{P-graph} if it satisfies the following:
	\begin{enumerate}
		\item every interior vertex has degree 4 and
		\item every face has even number of edges.
	\end{enumerate}
	If $D$ is a P-graph, there exists a \emph{P-labeling} $\mu: E_{int} \to \pm 1$ such that at each vertex, two opposite edges are labeled "+1" and the other two opposite edges are labeled "-1".
\end{definition}

The square lattice $\field{Z}^2$ is a P-graph. There is a P-labeling $\mu : E(\field{Z}^2) \to \pm 1$ that is $+1$ on ``horizontal" edges $\{(m,n),(m+1,n)\}$ and $-1$ on ``vertical" edges $\{(m,n),(m,n+1)\}$.

\begin{example}[Schramm's orthogonal circle patterns]
	We consider an orthogonal circle pattern in the complex plane with the combinatorics of a P-graph $D$: Each face of $D$ corresponds to a circle and neighboring faces correspond to two circles intersecting orthogonally. Furthermore, every vertex of $D$ is the intersection point of the neighboring circles, which yields a map $z:V \to \field{C}$.
	
	\begin{figure}[h]
		\centering
		\includegraphics[width=0.85\textwidth]{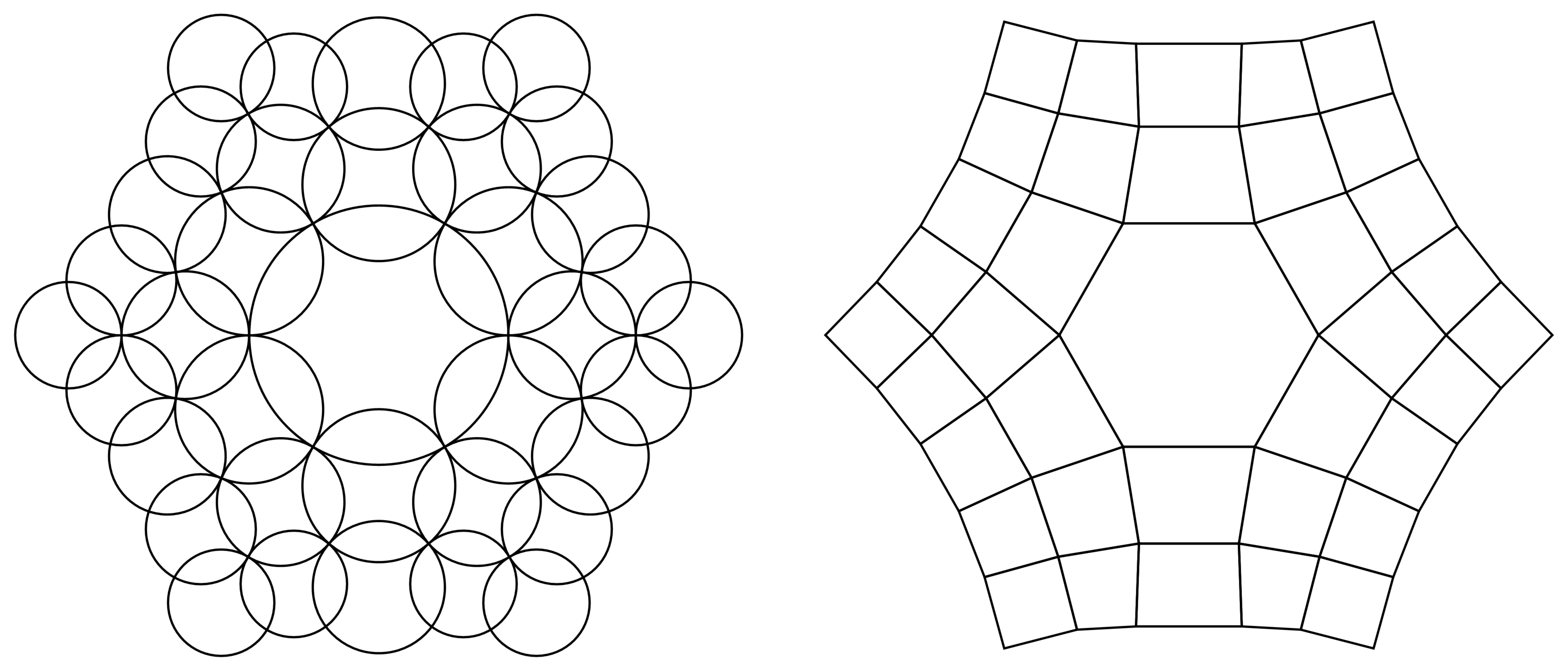}
		\caption{The intersection points of an orthogonal circle pattern (left) yields a P-net $z: V \to \field{C}$ (right) with its P-labeling $\mu:E_{int} \to \pm 1$ as a canonical discrete holomorphic quadratic differential.}
		\label{fig:planarcirc}
	\end{figure}
	
	Neighboring circles intersecting orthogonally implies that if one maps any interior vertex $v_0$ to infinity by inversion, the neighboring four vertices form a rectangle. In particular, a rectangle has opposite edges parallel. Denoting $v_1,v_2,v_3,v_4$ the neighboring vertices of $v_0$, we thus have
	\begin{equation*}
		\frac{1}{z_1 - z_0} - 	\frac{1}{z_2 - z_0} + 	\frac{1}{z_3 - z_0} - 	\frac{1}{z_4 - z_0} =0.
	\end{equation*}
	Hence the P-labeling $\mu$ is a discrete holomorphic quadratic differential. 
\end{example}

\begin{figure}[h]
	\centering
	\includegraphics[width=0.3\textwidth]{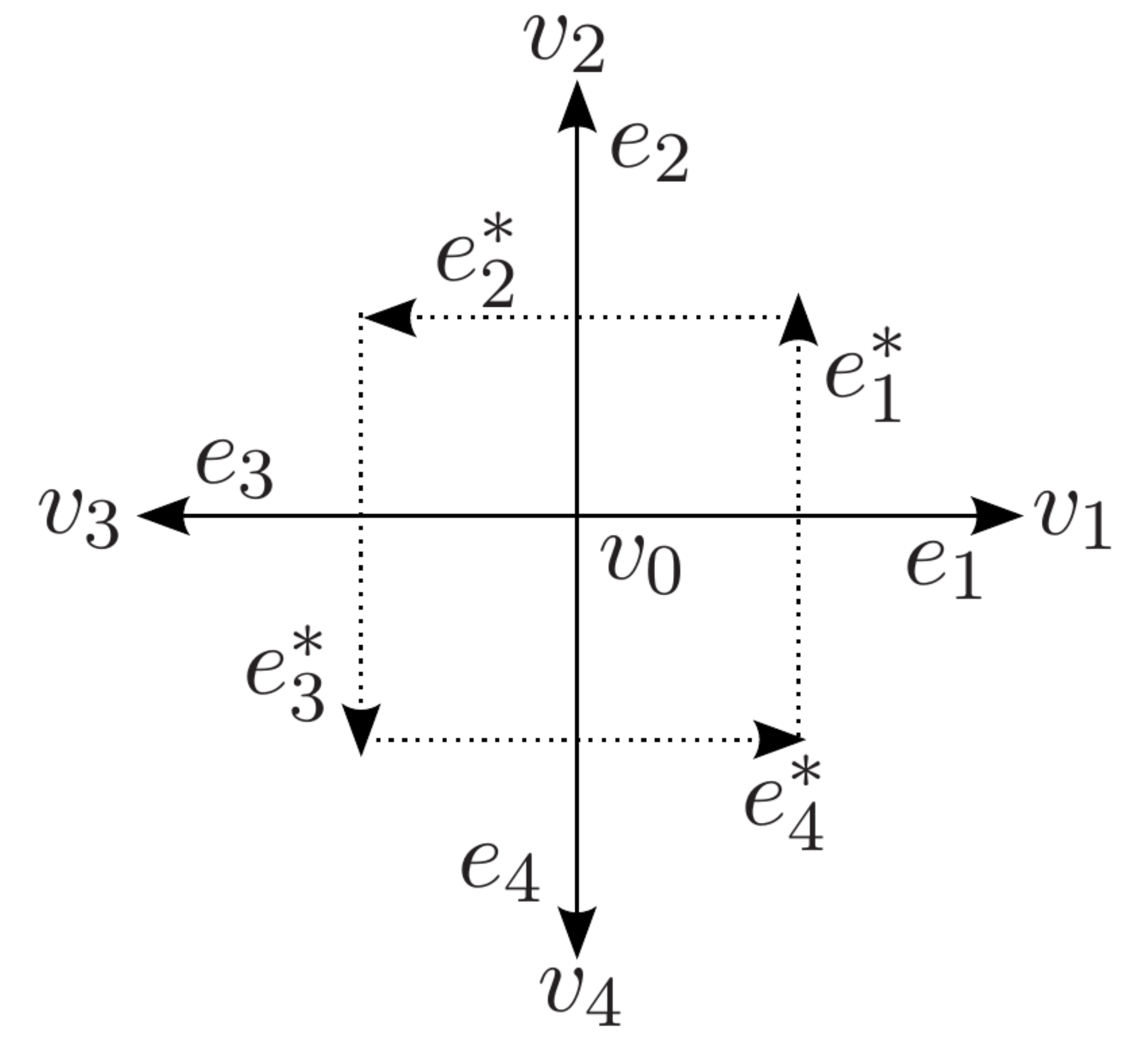}
	\caption{The edges of $D$ around a vertex $v_0$ are indicated by solid lines. The edges of the dual face $v_0^*$ are indicated by dotted lines.}
	\label{fig:indices}
\end{figure}

In the previous example, the necessary and sufficient condition for a P-labeling to be a discrete holomorphic quadratic differential is the \emph{parallelogram property}. It does not play any role whether faces are cyclic. 

\begin{definition}[\cite{Bobenko1999}] \label{def:parall}
	A non-degenerate realization $z:V \to \field{C}$ of a P-graph is a \emph{P-net} if it possesses the parallelogram property:
	Mapping any interior vertex $z_0:=z(v_0)$ to infinity by inversion, then the image of its four neighboring vertices $z_1,z_2,z_3,z_4$ form a parallelogram, i.e. 
	\begin{align} \label{eq:parallel}
		\frac{1}{z_1 - z_0} - 	\frac{1}{z_2 - z_0} + 	\frac{1}{z_3 - z_0} - 	\frac{1}{z_4 - z_0} =0.
	\end{align}
	(See Figure \ref{fig:indices} for the indices.) In particular, the P-labeling $\mu$ is a discrete holomorphic quadratic differential.
\end{definition}

P-nets were introduced by Bobenko and Pinkall \cite{Bobenko1999} while studying discrete integrable systems of isothermic surfaces. By replacing $1/(z_i-z_0)$ in Equation \eqref{eq:parallel} with $(z_i - z_0)/|z_i - z_0|^2$, P-nets can be defined in $\field{R}^n$.

\begin{example}\label{example:cr-1}
	By considering a discrete Lax representation, Bobenko and Pinkall \cite{Bobenko1996} introduced the notion of quadrilateral isothermic surfaces, which provides a discrete analogue of conformal curvature line parametrizations. We consider a quadrilateral isothermic surface $z: V(\field{Z}^2) \to \mathbb{C}$ in the complex plane, i.e. each face is a quadrilateral with cross-ratio
	\[
	cr(z_{m,n},z_{m+1,n},z_{m+1,n+1},z_{m,n+1}) =  -1 \quad \forall m,n \in \field{Z}.
	\]
	A crucial property is that there exists another quadrilateral isothermic surface $z^*: V(\field{Z}^2) \to \field{C}$ such that  
	\[
	z^*_{m+1,n} - z^*_{m,n} =  \frac{1}{\bar{z}_{m+1,n}-\bar{z}_{m,n}}, \quad
	z^*_{m,n+1} - z^*_{m,n} = - \frac{1}{\bar{z}_{m,n+1}-\bar{z}_{m,n}}.
	\]
	Furthermore, the diagonals of $z^*$ satisfy
	\begin{gather}
		z^*_{m+1,n} -z^*_{m,n+1} = \frac{2}{\bar{z}_{m+1,n+1}-\bar{z}_{m,n}}, \label{eq:diagonal1}\\
		z^*_{m+1,n+1} - z^*_{m,n} = \label{eq:diagonal2} \frac{2}{\bar{z}_{m+1,n}-\bar{z}_{m,n+1}} 
	\end{gather}
	as shown by \citet[Corollary 4.33]{Bobenko2008}. 
	
	Now, we consider half the vertices of a quadrilateral isothermic surface. We denote by $\field{Z}^2_b$, $\field{Z}^2_w$ the square lattices formed by the vertices
	\begin{gather*}
		V(\field{Z}^2_b):=\{(m,n)\in \field{Z}^2| m+n \text{ even}\}, \\
		V(\field{Z}^2_w):=\{(m,n)\in \field{Z}^2| m+n \text{ odd}\}.
	\end{gather*} 
	If $z:V(\field{Z}^2) \to \field{C}$ is an isothermic net, then \eqref{eq:diagonal1} and \eqref{eq:diagonal2} imply that for $(m,n)\in V(\field{Z}^2)$
	\begin{align*}
		0 =&\frac{1}{2}(\bar{z}^*_{m+1,n} -\bar{z}^*_{m,n+1} + \bar{z}^*_{m,n+1} -  \bar{z}^*_{m-1,n} +  \bar{z}^*_{m-1,n} -\bar{z}^*_{m,n-1} + \bar{z}^*_{m,n-1} -\bar{z}^*_{m+1,n}) \\
		=&\frac{1}{z_{m+1,n+1}-z_{m,n}} - \frac{1}{z_{m-1,n+1}-z_{m,n}} + \frac{1}{z_{m-1,n-1}-z_{m,n}} -\frac{1}{z_{m+1,n-1}-z_{m,n}}.
	\end{align*}
	Hence, $z^*|_{\field{Z}^2_w}$ and $z^*|_{\field{Z}^2_b}$ are P-nets and their P-labellings are discrete holomorphic quadratic differentials. 
\end{example}

In fact, the conditions for discrete holomorphic quadratic differentials are equivalent to the closedness of a certain $\field{C}^3$-valued dual 1-form, which yields a Weierstrass representation for discrete minimal surfaces. This formula was first derived from the $sl(2,\field{C})$-formulation of infinitesimal conformal deformations by \citet{Lam2015a}.

\begin{lemma}\label{lem:weierstrass}
	Suppose $z:V \to \field{C}$ is a non-degenerate realization of a simply connected discrete surface. A function $q:E_{int} \to \field{R}$ is a discrete holomorphic quadratic differential if and only if there exists $\mathcal{F}:V^* \to \field{C}^3$ such that for every edge $\{ij\} \in E_{int}$
	\begin{equation*}
		d\mathcal{F}(e^*_{ij}) =\frac{q_{ij}}{z_j -z_i} \left( \begin{array}{c}
			1-z_i z_j \\ i(1+z_i z_j) \\ z_i + z_j
		\end{array}\right).
	\end{equation*}
\end{lemma}
\begin{proof}
	Given a discrete holomorphic quadratic differential $q$, we have 
	\[
	\sum_j q_{ij} z_i z_j /(z_j -z_i) = z_i \sum_j q_{ij} + z^2_i \sum_j q_{ij} /(z_j -z_i) =0.
	\]
	Together with the definition of discrete holomorphic quadratic differentials, it implies that the $\field{C}^3$-valued dual 1-form
	\[
	\eta(e_{ij}^*) := \frac{q_{ij}}{z_j -z_i} \left( \begin{array}{c}
	1-z_i z_j \\ i(1+z_i z_j) \\ z_i + z_j
	\end{array}\right)
	\]
	satisfies
	\[
	\sum_j \eta(e_{ij}^*) = \sum_j \frac{q_{ij}}{z_j -z_i} \left( \begin{array}{c}
	1-z_i z_j \\ i(1+z_i z_j) \\ z_i + z_j
	\end{array}\right) =0
	\]
	and hence is closed.
	Because the discrete surface is simply connected, $\eta$ is exact and there exists $\mathcal{F}$ such that
	\[
	d\mathcal{F}(e^*_{ij}) = \eta(e_{ij}^*).
	\]
	The converse can be shown directly.
\end{proof}

As expected from the smooth theory, the discrete surfaces given by
\[
\Re(e^{i\theta}\mathcal{F}) \quad \forall \,\theta \in [0,2\pi]
\]
should form an associated family of discrete minimal surfaces. Furthermore, $\Re(\mathcal{F})$ is analogous to the asymptotic line parametrization of a minimal surface while $\Re(i\mathcal{F})$ is analogous to the curvature line parametrization.

We denote by $N:V \to \field{S}^2$ the inverse stereographic projection of $z$, i.e.
\begin{equation*}
	N := \frac{1}{1+|z|^2} \left(\begin{array}{c}2\Re z \\ 2\Im z \\ |z|^2-1\end{array}\right).		
\end{equation*} 
We get
\begin{align} \label{eq:ref}
	\Re (d\mathcal{F}(e^*_{ij})) &= \frac{q_{ij}(1+|z_i|^2) (1+|z_j|^2)}{2|z_j-z_i|^2}  (N_j - N_i), \\ \label{eq:imf}
	\Re (i\, d\mathcal{F}(e^*_{ij}))&= \frac{q_{ij}(1+|z_i|^2) (1+|z_j|^2)}{2|z_j-z_i|^2}  (N_i \times N_j)
\end{align}
where $\times$ denotes the cross product in Euclidean space.

On one hand, the realization $\Re (\mathcal{F})$ is a \emph{reciprocal parallel mesh} of $N$, i.e. $\Re (\mathcal{F})$ is a combinatorial dual to $N$ and their corresponding edges are parallel. The realization $\Re (\mathcal{F})$ satisfies the definition of discrete minimal surfaces in \cite{Lam2015} as a Christoffel dual of an inscribed discrete surface. Furthermore, since a face $i^*$ of $\Re (\mathcal{F})$ is planar if and only if the vertex $i$ of $N$ has planar vertex star, we observe that the faces of $\Re (\mathcal{F})$ are non-planar in general.

In contrast, $\Re (i \mathcal{F})$ always has planar faces since for every dual face $i^* \in F^*$
\[
\langle N_i, \Re (i\,d\mathcal{F}(e^*_{ij})) \rangle = 0.
\] 
In particular $N_i$ is the normal of face $i^*$ of $\Re (i \mathcal{F})$.

In the next section, we focus on the discrete surfaces $\Re ( \mathcal{F})$ and $\Re (i \mathcal{F})$. As we will see they unify the earlier approaches to discrete minimal surfaces.

\section{Discrete minimal surfaces and their conjugates} \label{sec:dismin}

\subsection{A-minimal surfaces}

We introduce two types of discrete minimal surfaces. We call the first type as A-minimal surfaces, which was considered in \cite{Lam2015} as the Christoffel duals of the Gauss maps. As illustrated in Example \ref{example:cir}, the edges of A-minimal surfaces are in \emph{asymptotic line directions}. 

\begin{definition}
	Given a discrete surface $M$ and its dual $M^*$, a realization $f:V^* \to \field{R}^3$ of $M^*$ is \emph{A-minimal} with Gauss map $N:V \to \field{S}^2 \subset \mathbb{R}^3$ if $N$ is admissible and for every interior edge $\{ij\}$
	\begin{align}
		dN(e_{ij}) \times df(e^*_{ij}) =0 \label{eq:anet1}, \\
		\langle N_i + N_j, df(e_{ij}^*) \rangle =0 \label{eq:anet2}
	\end{align}
	where $ \times $ denotes the cross product in $\mathbb{R}^3$.
\end{definition}
\begin{remark}
	If $N$ is non-degenerate, i.e. $N_i \neq N_j$, then \eqref{eq:anet1} implies \eqref{eq:anet2}.
\end{remark}
\begin{remark}
	If $f$ is A-minimal and $v \in V^*$ is a vertex of degree three, the image of $v$ and its three neighboring vertices lie on an affine plane. It is closely related to Bobenko and Pinkall's discretization of asymptotic line parametrizations \cite{Bobenko1999}.
\end{remark}
For A-minimal surfaces in the following example, their edges are in \emph{asymptotic line directions}.
\begin{example}[Circular minimal surfaces \cite{Bobenko1996}] \label{example:cir}
	Here we consider the general form of quadrilateral isothermic surfaces, which provides a discretization of principal curvature lines of a smooth surface \cite{Bobenko1996}. Given any four distinct points in space, we identify a sphere containing the four points as a complex plane via stereographic projection and consider the complex cross ratio of the four points. Such a complex number is well defined up to conjugation. Particularly if the cross ratio is real, the four points are concyclic. A map $f: V(\field{Z}^2) \to \field{R}^3$ is a quadrilateral isothermic surface if every elementary quadrilateral is cyclic and has factorized real cross-ratios
	\[
	cr(f_{m,n},f_{m+1,n},f_{m+1,n+1},f_{m,n+1}) = \frac{\alpha_m}{\beta_n} \quad \forall m,n \in \field{Z},
	\]
	where $\alpha_m \in \field{R}$ does not depend on $n$ and $\beta_n \in \field{R}$ not depend on $m$. Then there exists another discrete isothermic net $f^*: \field{Z}^2 \to \field{R}^3$ satisfying  
	\[
	f^*_{m+1,n} - f^*_{m,n} = \alpha_m \frac{f_{m+1,n}-f_{m,n}}{||f_{m+1,n}-f_{m,n}||^2}, \quad
	f^*_{m,n+1} - f^*_{m,n} =\beta_n \frac{f_{m,n+1}-f_{m,n}}{||f_{m,n+1}-f_{m,n}||^2}.
	\]
	In fact every elementary quadrilateral of $f^*$ is cyclic. If $f(\field{Z}) \subset \field{S}^2$, then $f^*$ is called a circular minimal surface.
	
	As in Example \ref{example:cr-1} we consider half the vertices of a quadrilateral isothermic surface. If $f:\field{Z}^2 \to \field{S}^2$ is a quadrilateral isothermic surface, then 
	\[
	f^*|_{\field{Z}^2_w} \text{ is A-minimal with Gauss map } f|_{\field{Z}^2_b}.
	\]
	Here the edges of $f^*|_{\field{Z}^2_w}$ are diagonals of the quadrilaterals of $f$ and hence are regarded as a discretization of asymptotic lines. Furthermore, one can obtain A-minimal surfaces from a circular minimal surface by adding diagonals arbitrarily.  
\end{example}

\subsection{C-minimal surfaces}

We introduce another type of discrete minimal surfaces, C-minimal surfaces, which reflect the property that smooth minimal surfaces have vanishing mean curvature. As illustrated in Example \ref{example:conical}, their edges are in the principal directions, which yields a discrete analogue of \emph{curvature line parametrization}.  

Suppose we have a realization $\tilde{f}:V^* \to \field{R}^3$ of $M^*$ with planar faces. We pick a normal $N:F^* \to \field{S}^2$ for each face such that $N$ is admissible, i.e. $N_i \neq -N_j$ for every edge $\{ij\}$. We then measure its dihedral angles. If $d\tilde{f}(e^*_{ij}) \neq 0$, the sign of the dihedral angle $\alpha_{ij} \in (-\pi,\pi)$ is determined by
\begin{align*}
	\sin \alpha_{ij} &= \langle N_i \times N_j, \frac{d\tilde{f}(e^*_{ij})}{|d\tilde{f}(e^*_{ij})|} \rangle, \\
	\cos  \alpha_{ij} &= \langle N_i,N_j \rangle 
\end{align*}
where $i^*,j^*\in F^*$ denote the left and the right face of $e^*_{ij}$. In the following, we are interested in the quantity $|d\tilde{f}| \tan (\alpha/2)$ defined on edges. If an edge degenerates, though there is an ambiguity in the sign of the dihedral angle, the quantity $|d\tilde{f}| \tan (\alpha/2)$ vanishes.

\begin{definition} \label{def:c-minimal}
	Given a discrete surface $M$ and its dual $M^*$, a realization $\tilde{f}:V^* \to \field{R}^3$ of $M^*$ is \emph{C-minimal} with Gauss map $N:V \to \field{S}^2$ if $\tilde{f}$ has planar faces with face normal $N$ admissible and the \emph{scalar mean curvature} $\tilde{H}: F^*_{int} \to \field{R}$ defined by
	\[
	\tilde{H}_{i} := \sum_j |d\tilde{f}(e^*_{ij})| \tan \frac{\alpha_{ij}}{2} \quad \forall i \in V_{int} \cong F^*_{int}
	\]
	vanishes identically.
\end{definition}

We rewrite the scalar mean curvature in terms of face normals.
\begin{lemma}\label{lem:dotmean}
	Suppose $\tilde{f}:V^* \to \field{R}^3$ is a realization of a discrete surface $M^*$ with face normal $N:V \to \field{R}^3$.
	If $N_i \neq \pm N_j$ on edge $\{ij\}$, then 
	\[
	d\tilde{f}(e_{ij}^*) = k_{ij} N_i \times N_j
	\]
	for some $k_{ij}=k_{ji} \in \field{R}$ and
	\[
	|d\tilde{f}(e^*_{ij})| \tan \frac{\alpha_{ij}}{2} = k_{ij} (1- \langle N_i , N_j \rangle).
	\]
\end{lemma}
\begin{proof}
	Suppose $\{ij\}$ is an edge of $M$ with $N_i \neq \pm N_j$. Since $N_i,N_j \perp d\tilde{f}(e^*_{ij})$ there exists $k_{ij} \in \field{R}$ such that
	\[
	d\tilde{f}(e^*_{ij}) = k_{ij} N_i \times N_j.
	\]
	The property $d\tilde{f}(e^*_{ji}) = -d\tilde{f}(e^*_{ij})$ implies $k_{ij}=k_{ji}$.
	If $d\tilde{f}(e_{ij}^*)=0$, then $k_{ij}=0$ and
	\[
	|d\tilde{f}(e^*_{ij})| \tan \frac{\alpha_{ij}}{2} =0= k_{ij} (1- \langle N_i , N_j \rangle).
	\]
	If $d\tilde{f}(e_{ij}^*)\neq 0$, the dihedral angle $\alpha_{ij}$ satisfies
	\[
	\sin \alpha_{ij} = \langle N_i \times N_j, \frac{d\tilde{f}(e^*_{ij})}{|d\tilde{f}(e^*_{ij})|} \rangle = \sign(k_{ij}) |\sin \alpha_{ij}| 
	\]
	and hence
	\begin{align*}
		k_{ij} (1- \langle N_i, N_j \rangle) = \frac{|k_{ij} N_i \times N_j|}{\sin \alpha_{ij}} \cdot 2\sin^2 \frac{\alpha_{ij}}{2} = |d\tilde{f}(e^*_{ij})| \tan \frac{\alpha_{ij}}{2}.
	\end{align*} 
\end{proof}

In the following, we see that conical minimal surfaces, as a discrete analogue of curvature line parametrizations \cite{Bobenko2007,Pottmann2008}, are C-minimal surfaces.
\begin{example}[Conical minimal surfaces] \label{example:conical}
	A realization $\tilde{f}:V^* \to \field{R}^3$ of a discrete surface $M$ with planar faces and non-degenerate face normal $N:V \to \field{S}^2$ is a \emph{conical} surface if $N$ has planar faces. In this case, the poles of the faces of $N$ with respect to the unit sphere yield a realization $\hat{N}: V^* \to \field{R}^3$ of $M^*$ with planar faces tangent to the unit sphere: for each dual face $i^* = (\phi^*_1,\phi^*_2,\dots,\phi^*_n) \in F^*$
	\[
	\langle N_i, \hat{N}_{\phi_r} \rangle =1.
	\]
	For $t \in \field{R}$ the area of each planar face under a face offset $\tilde{f} + t \hat{N}$ is
	\begin{align*}
		\Area(\tilde{f} + t\hat{N})_i & =  \frac{1}{2} \sum_r \langle (\tilde{f}_{\phi_r} + t \hat{N}_{\phi_r}) \times (\tilde{f}_{\phi_{r+1}} + t \hat{N}_{\phi_{r+1}}), N_i \rangle \\
		&=: \Area(\tilde{f})_i + \Area(\tilde{f},\hat{N})_i t + \Area(\hat{N})_i t^2 .
	\end{align*}
	\citet{Bobenko2010a} considered conical surfaces with vanishing mixed area $\Area(\tilde{f},\hat{N})\equiv 0$ as \emph{conical minimal surfaces} \cite{Bobenko2010a, Muller2015}. Karpenkov and Wallner \cite{Karpenkov2014} showed that mixed area of conical surfaces coincides with scalar mean curvature defined in Definition \ref{def:c-minimal}: We write $\tilde{f}_{\phi_{r+1}}- \tilde{f}_{\phi_r}=k_{ir} N_i \times N_r$ and denote by $\phi_r,\phi_{r+1}$ the right and the left face of $e_{ir}$. We have
	\[
	\langle \hat{N}_{\phi_r}, N_i \rangle = \langle \hat{N}_{\phi_{r+1}}, N_i \rangle = \langle \hat{N}_{\phi_r}, N_r \rangle = \langle \hat{N}_{\phi_{r+1}}, N_r \rangle = 1.
	\]
	Then for every interior vertex $i$
	\begin{align*}
		\Area(\tilde{f},\hat{N})_i &= \frac{1}{2} \sum_r \langle \tilde{f}_{\phi_r} \times \hat{N}_{\phi_{r+1}} + \hat{N}_{\phi_r} \times \tilde{f}_{\phi_{r+1}}, N_i \rangle \\
		&= \frac{1}{2} \sum_r \langle (\hat{N}_{\phi_r} + \hat{N}_{\phi_{r+1}}) \times (\tilde{f}_{\phi_{r+1}}- \tilde{f}_{\phi_r}), N_i \rangle \\
		&= \frac{1}{2} \sum_r \langle \hat{N}_{\phi_r} + \hat{N}_{\phi_{r+1}},   k_{ir} (N_r - \langle N_i, N_r \rangle N_i )\rangle \\
		&= \sum_r k_{ir} (1 - \langle N_i, N_r \rangle) \\
		&= \sum_r |d\tilde{f}(e^*_{ir})| \tan \frac{\alpha_{ir}}{2}.
	\end{align*}
	Hence, conical minimal surfaces are C-minimal.  
\end{example}

\subsection{Conjugate minimal surfaces}
We show that there is a one-to-one correspondence between A-minimal surfaces and C-minimal surfaces (Figure \ref{fig:conjugate}).
\begin{figure}[h]
	\centering
	\includegraphics[width=0.9\textwidth]{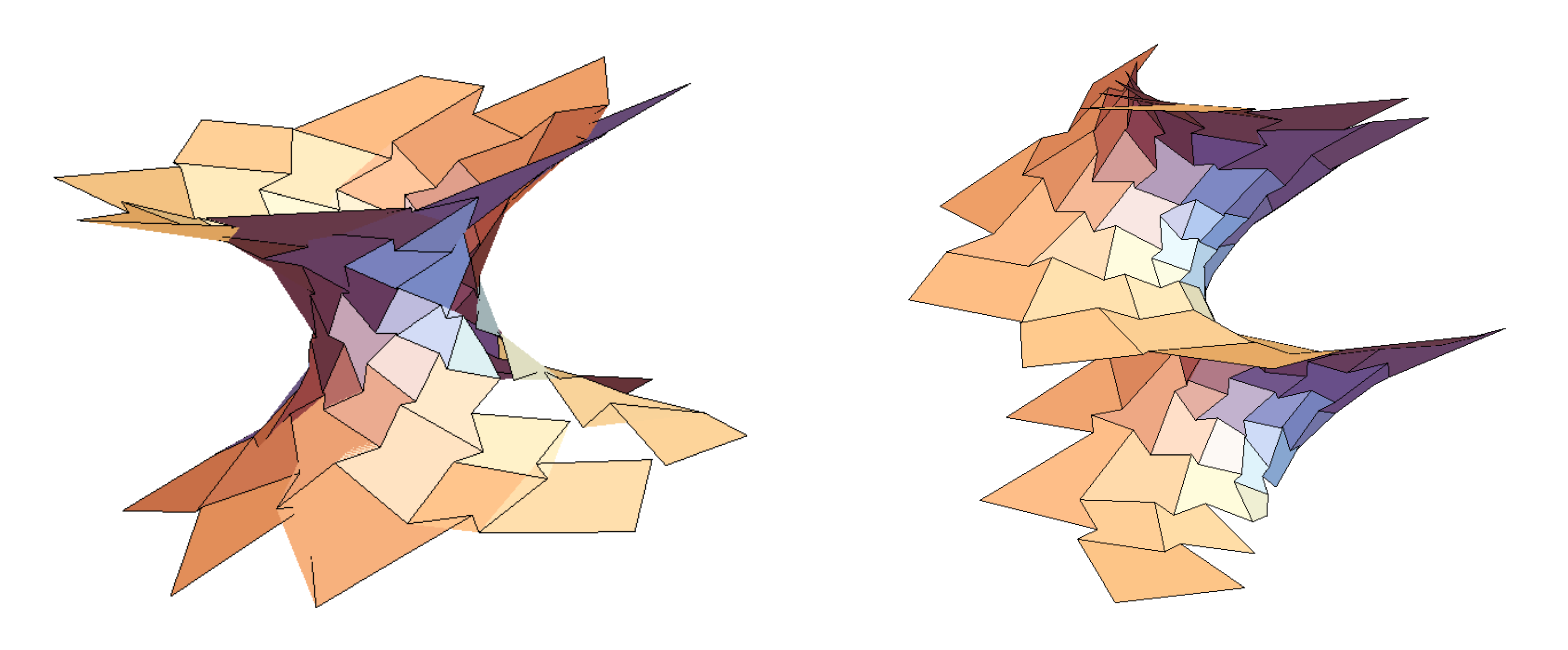}
	\caption{A conjugate pair of discrete minimal surfaces. Left: An A-minimal surface, trivalent with planar vertex stars. It is a reciprocal parallel mesh of a triangulated surface inscribed in the unit sphere. Right: An C-minimal surface, trivalent with planar faces. On each face $\phi$, the scalar mean curvature $\sum_{e \in \phi} \ell_e \tan \frac{\alpha_e}{2}$ vanishes. Here $\ell$ denotes the edge lengths and $\alpha$ denotes the dihedral angles.}
	\label{fig:conjugate}
\end{figure}

\begin{theorem} \label{thm:conjugate}
	Let $M$ be a simply connected discrete surface and $M^*$ be its dual. For every admissible realization $N:V \to \field{S}^2$ of $M$, each A-minimal surface $f:V^* \to \field{R}^3$ with Gauss map $N$ yields a C-minimal surface $\tilde{f}:V^* \to \field{R}^3$ with Gauss map $N$ via
	\[
	d\tilde{f}(e_{ij}^*)= N_i \times df(e_{ij}^*) = N_j \times df(e_{ij}^*)
	\]
	and vice versa. We say $(f,\tilde{f})$ form a conjugate pair of minimal surfaces. 
\end{theorem}
\begin{proof}
	Suppose $f:V^* \to \field{R}^3$ is A-minimal with Gauss map $N:V \to \field{S}^2$. Then there exists a well defined dual 1-form $\eta: \vec{E}^*_{int} \to \field{R}^3$ on $M^*$ such that
	\begin{gather} \label{eq:etaaa}
		\eta(e^*_{ij}) := N_j \times df(e^*_{ij}) = N_i \times df(e^*_{ij}) 
	\end{gather}
	and for every interior vertex $i$
	\begin{gather}
		\sum_j \eta(e^*_{ij}) = N_i \times \sum_j df(e^*_{ij}) =0, \label{eq:conjclosed}\\
		\langle N_i, \sum_j df(e^*_{ij}) \rangle =0 \label{eq:zeromeancurv}.
	\end{gather}
	Since $M$ is simply connected, the closedness condition in \eqref{eq:conjclosed} implies $\eta$ is exact. Hence there exists $\tilde{f}: V^* \to \field{R}^3$ such that for every interior oriented edge $e_{ij}^*\in \vec{E}_{int}$ 
	\[
	d\tilde{f}(e_{ij}^*) = \eta(e_{ij}^*).
	\]
	We now show that $\tilde{f}$ is C-minimal with Gauss map $N$. 
	
	Firstly, equation \eqref{eq:etaaa} implies $\tilde{f}$ has planar faces with face normal $N$.
	
	Secondly, if $N_i = N_j$ for some interior edge $\{ij\}$ then
	\[
	\langle N_i, df(e^*_{ij}) \rangle = \frac{1}{2}\langle N_i+N_j, df(e^*_{ij}) \rangle = 0.
	\]
	If $N_i \neq N_j$, we write $df(e^*_{ij}) = k_{ij} (N_j -N_i)$ for some $k_{ij}=k_{ji} \in \field{R}$ and thus $d\tilde{f}(e^*_{ij}) = k_{ij} N_i \times N_j$. Lemma \ref{lem:dotmean} implies
	\begin{align*}
		\langle N_i, \sum_j df(e^*_{ij}) \rangle &= \sum_{j|N_i\neq N_j} k_{ij} (\langle N_i,N_j \rangle -1)\\ &= -\sum_{j|N_i\neq N_j} |d\tilde{f}(e^*_{ij})| \tan \frac{\alpha_{ij}}{2}\\ &= -\sum_{j} |d\tilde{f}(e^*_{ij})| \tan \frac{\alpha_{ij}}{2}.
	\end{align*}
	Hence \eqref{eq:zeromeancurv} implies that $\tilde{f}$ has vanishing scalar mean curvature and thus is C-minimal with Gauss map $N$.
	
	Conversely, we suppose $\tilde{f}$ is C-minimal with Gauss map $N$ and define a dual 1-form $\omega: \vec{E}_{int}^*\to \field{R}^3$ by
	\[
	\omega(e^*_{ij}) := -N_i \times d\tilde{f}(e_{ij}^*) -  |d\tilde{f}(e^*_{ij})| \tan \frac{\alpha_{ij}}{2}\,N_i .
	\] To check $\omega$ is a 1-form on $M^*$, note that if $N_i =N_j$ then $\alpha_{ij}=0$ and
	\[
	\omega(e^*_{ji}) = -N_j \times d\tilde{f}(e_{ji}^*) = N_i \times d\tilde{f}(e_{ij}^*) = -\omega(e^*_{ij}).
	\]
	If $N_i \neq N_j$, writing $d\tilde{f}(e_{ij}^*) = k_{ij} N_i \times N_j$ yields
	\[
	\omega(e^*_{ij}) = -N_i \times (k_{ij} N_i \times N_j) -k_{ij} (1-\langle N_i,N_j \rangle)N_i  = k_{ij} (N_j - N_i)
	\]
	and hence $\omega(e^*_{ji}) = -\omega(e^*_{ij})$.
	
	In addition $\tilde{f}$ being C-minimal implies that $\omega$ is a closed dual 1-form on $M^*$
	\[
	\sum_j \omega(e^*_{ij}) =0.
	\]
	Since $M^*$ is simply connected, there exists $f:V^* \to \field{R}^3$ such that
	\[
	df(e^*_{ij}) = \omega(e^*_{ij}) \quad \forall \{ij\} \in E_{int}
	\]
	and $f$ is A-minimal with Gauss map $N$.
\end{proof}

\begin{example}[Cubic polyhedra] A cubic polyhedron is a polyhedral surface with edges exactly the same as those of the cubic lattice. Goodman-Strauss and Sullivan \cite{Strauss2003} showed that they are discrete minimal surfaces in the sense of the variational approach \cite{Pinkall1993}, which provides a discretization of triply periodic minimal surfaces. A discretization of Schwarz CLP surface and Schwarz P surface are shown in Figure \ref{fig:cubic}. Since a cubic polyhedron has convex faces, there exists a canonical choice of normal vectors determined by the orientation. Its edge lengths are all equal and its dihedral angles are either $\frac{\pi}{2}$, $0$ or $-\frac{\pi}{2}$. One can check that they are C-minimal. Following from Theorem \ref{thm:conjugate} we know that they have conjugate surfaces. In particular, the A-minimal surface conjugate to the discrete Schwarz P surface (Figure \ref{fig:cubic} right) consists of edges that are exactly asymptotic lines of the smooth Schwarz D surface.  
	\begin{figure}[h]
		\centering
		\includegraphics[width=0.9\textwidth]{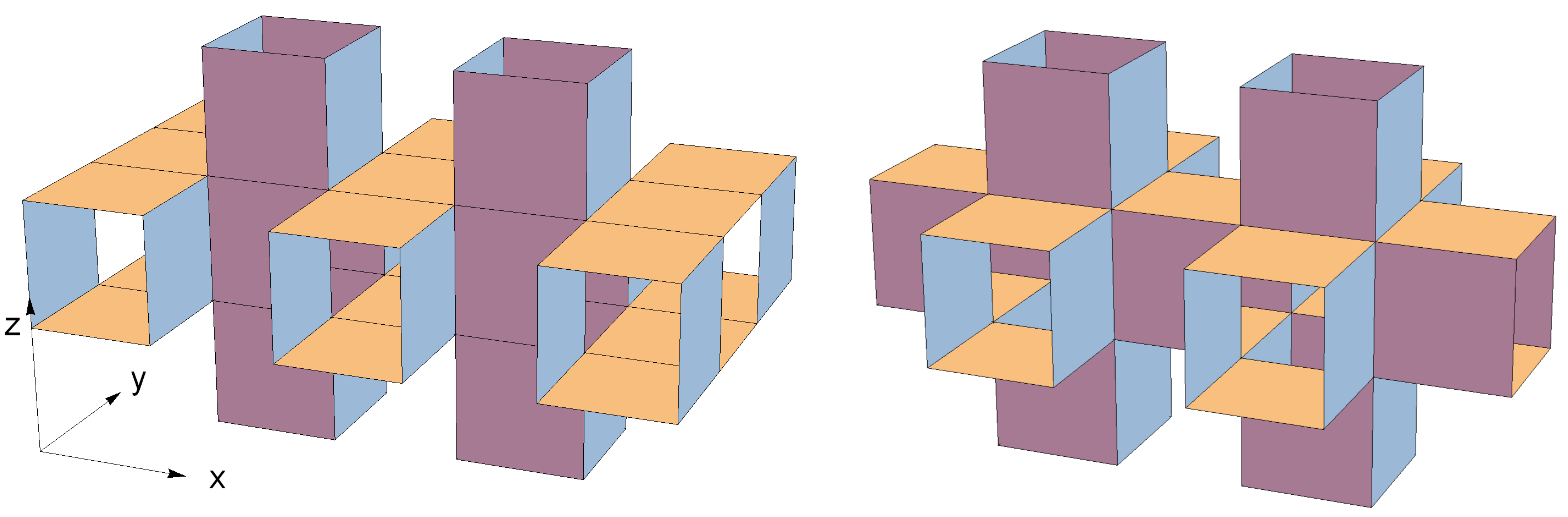}
		\caption{Parts of two C-minimal symmetric cubic polyhedra. Left: Merging two successive squares parallel to xy-plane and xz-plane respectively into rectangles yields vanishing scalar mean curvature. Right: Each square face has vanishing scalar mean curvature.}
		\label{fig:cubic}
	\end{figure} 
\end{example}

Since we have a notion of conjugate minimal surfaces sharing the same combinatorics, we can consider the associated family of minimal surfaces as in the smooth theory. As we will see in Section \ref{sec:criticalarea}, they share certain common properties.

\begin{definition}\label{def:assoicatedfamily}
	For every conjugate pair of minimal surfaces $(f,\tilde{f})$ we define its \emph{associated family of surfaces} as a $S^1$-family of realizations $f^{\theta}:V^* \to \field{R}^3$ 
	\[
	f^{\theta}:= \cos \theta f + \sin \theta \tilde{f}
	\]
	for $\theta \in [0,2\pi]$.
\end{definition}

\subsection{Scalar mean curvature on faces}

Mean curvature was introduced in \cite{Schief2004,Schief2007,Bobenko2010a} on polyhedral surfaces with vertex normals by means of Steiner's formula and further extended to the associated family in \cite{Hoffmann2014}.

Recall that C-minimal surfaces are defined by vanishing scalar mean curvature. In fact, the notion of scalar mean curvature can be extended naturally to its associated family of discrete surfaces. 

\begin{proposition}
	Every A-minimal surface $f$ with an admissible Gauss map $N:V \to \field{S}^2$ satisfies
	\begin{align*}
		\langle dN(e_{ij}) \times df(e^*_{ij}), \frac{N_i+N_j}{|N_i+N_j|^2} \rangle =0  \quad  &\forall \{ij\} \in E_{int}, \\
		\sum_j \langle dN(e_{ij}), df(e^*_{ij}) \rangle =0  \quad &\forall i\in V_{int}.
	\end{align*}
	On the other hand, every C-minimal surface  $\tilde{f}: V^* \to \field{R}^3$ with Gauss map $N$ satisfies
	\begin{align*}
		\sum_j  \langle dN(e_{ij}) \times d\tilde{f}(e^*_{ij}), \frac{N_i+N_j}{|N_i+N_j|^2} \rangle =0  \quad &\forall i\in V_{int},  \\
		\langle dN(e_{ij}), d\tilde{f}(e^*_{ij}) \rangle =0  \quad &\forall \{ij\} \in E_{int}.
	\end{align*}
\end{proposition}

\begin{proof}
	Since $f$ is A-minimal with Gauss map $N$, on every interior edge $\{ij\}$ we have
	\[
	dN(e_{ij}) \times df(e^*_{ij}) = 0
	\]
	and hence
	\[
	\langle dN(e_{ij}) \times df(e^*_{ij}), \frac{N_i+N_j}{|N_i+N_j|^2} \rangle =0.
	\]
	We consider an interior vertex $i$ and for each of its neighboring vertex $j$ we write $df(e_{ij}^*) = k_{ij}( N_j-N_i) $ whenever $N_i \neq N_j$. Then the dual face $i^*$ being a closed polygon 
	\[
	\sum_j df(e^*_{ij}) =0
	\]
	implies
	\begin{align*}
		\sum_j \langle dN(e_{ij}), df(e^*_{ij}) \rangle =2 \sum_{j|N_i \neq N_j} k_{ij} \langle N_i - N_j , N_i \rangle= -2\sum_j \langle df(e^*_{ij}), N_i \rangle =0
	\end{align*}
	where we used $\langle df(e^*_{ij}), N_i \rangle =0$ if $N_i = N_j$. 
	
	On the other hand we know
	\[
	N_i, N_j \perp d\tilde{f}(e_{ij}^*)
	\]
	and hence
	\[
	\langle dN(e_{ij}), d\tilde{f}(e^*_{ij}) \rangle =0.
	\]
	For every edge $\{ij\}$ such that $N_i \neq N_j$, we have $d\tilde{f}(e_{ij}^*) = N_i \times df(e_{ij}^*) = k_{ij} N_i \times N_j$. Hence
	\begin{align*}
		\sum_j  \langle dN(e_{ij}) \times d\tilde{f}(e^*_{ij}), \frac{N_i+N_j}{|N_i+N_j|^2} \rangle &= \sum_{j|N_i \neq N_j}  k_{ij} \langle N_i-N_j, N_i \rangle \langle N_i + N_j, \frac{N_i+N_j}{|N_i+N_j|^2} \rangle \\
		&= \sum_j |d\tilde{f}(e^*_{ij})| \tan \frac{\alpha_{ij}}{2} \\
		&=0.
	\end{align*}
\end{proof}

We can extend the equations in the proposition above to surfaces in the associated family.
\begin{corollary}
	If $(f,\tilde{f})$ form a conjugate pair of minimal surfaces, then the discrete surfaces $f^{\theta}:= \cos \theta f + \sin \theta \tilde{f}$ in the associated family satisfy for any interior vertex $i\in V_{int}$ (i.e. any face of $f^{\theta}$)
	\begin{align}
		H^{\theta}_i:=\sum_j  \langle dN(e_{ij}) \times df^{\theta}(e^*_{ij}), \frac{N_i+N_j}{|N_i+N_j|^2} \rangle &=0,  \label{eq:vanmeancurfac} \\
		\sum_j \langle dN(e_{ij}), df^{\theta}(e^*_{ij}) \rangle & =0 \label{eq:vandotfac}.
	\end{align}
	Here $H^{\theta}:F^*_{int} \to \field{R}$ is called the \emph{scalar mean curvature} of $f^{\theta}$. Furthermore $H^{\pi/2}$ coincides with the scalar mean curvature $\tilde{H}$ of the C-minimal surface $\tilde{f}$. 
\end{corollary}

The following remark explains the smooth counterpart of \eqref{eq:vanmeancurfac} and \eqref{eq:vandotfac}.

\begin{remark}\label{rmk:sumsmooth}
	Suppose $f:M \to \field{R}^3$ is a smoothly immersed surface with Gauss map $N$ and $X_1,X_2 \in T_p M$ form an orthonormal basis in principal directions. Then the corresponding principal curvatures $\kappa_1,\kappa_2 \in \field{R}$ satisfy
	\[
	dN_p(X_i) = \kappa_i df_p(X_i). 
	\]
	We denote by $J$ the almost complex structure induced via
	\[
	N\times df(\cdot) = df(J\cdot).
	\] 
	For every $\theta \in \field{R}$ we have
	\begin{align*}
		\langle dN_p(\cos \theta X_1 + \sin \theta X_2) \times df_p(J_p(\cos \theta X_1 + \sin \theta X_2)),N_p \rangle = H + \frac{\kappa_1-\kappa_2}{2} \cos 2\theta, \\
		\langle dN_p(\cos \theta X_1 + \sin \theta X_2), df_p(J_p(\cos \theta X_1 + \sin \theta X_2))\rangle =  \frac{\kappa_1-\kappa_2}{2} \sin 2\theta 
	\end{align*}
	where $H= \frac{\kappa_1+\kappa_2}{2}$ is the mean curvature. Averaging over all directions yields
	\begin{align*}
		\frac{1}{2\pi} \int_0^{2\pi}\langle dN_p(\cos \theta X_1 + \sin \theta X_2) \times df_p(J_p(\cos \theta X_1 + \sin \theta X_2)) ,N_p \rangle d\theta= H, \\
		\frac{1}{2\pi} \int_0^{2\pi}\langle dN_p(\cos \theta X_1 + \sin \theta X_2), df_p(J_p(\cos \theta X_1 + \sin \theta X_2))\rangle d\theta=  0.
	\end{align*}
\end{remark}

\section{Weierstrass representation}\label{sec:Weierstrass}

In this section, we prove the complete version of the Weierstrass representation for discrete minimal surfaces. We show that it leads to Goursat transformations of discrete minimal surfaces. 

\begin{theorem} \label{thm:Weierstrass}
	Suppose $z:V \to \field{R}^2 \cong \field{C}$ is a non-degenerate realization of a simply connected surface and $q:E_{int} \to \field{R}$ is a discrete holomorphic quadratic differential. Then there exists $\mathcal{F}:V^* \to \field{C}^3$ such that for every edge $\{ij\} \in E_{int}$
	\begin{equation} \label{eq:eta}
		d\mathcal{F}(e^*_{ij}) =\frac{q_{ij}}{z_j - z_i} \left( \begin{array}{c}
			1-z_i z_j \\ i(1+z_i z_j) \\ z_i + z_j
		\end{array}\right).
	\end{equation}
	We assume the inverse stereographic projection $N:V \to \field{S}^2$ of $z$ given by
	\begin{equation*}
		N := \frac{1}{1+|z|^2} \left(\begin{array}{c}2\Re z \\ 2\Im z \\ |z|^2-1\end{array}\right)		
	\end{equation*}
	is admissible, i.e. $N_i \neq -N_j$ for $\{ij\} \in E$. We then have
	\begin{enumerate}
		\item $f:=\Re(\mathcal{F}):V^* \to \field{R}^3$ is A-minimal and
		\item $\tilde{f}:=\Re(i\mathcal{F}):V^* \to \field{R}^3$ is C-minimal.
	\end{enumerate}
	The realizations $(f,\tilde{f})$ form a conjugate pair with Gauss map $N$.
	
	The converse also holds: For every conjugate pair of minimal surfaces $f,\tilde{f}$ there exists a holomorphic quadratic differential $q$ on the stereographic projection of their Gauss map such that $\mathcal{F}:= f- i\tilde{f}$ satisfies \eqref{eq:eta}.
\end{theorem}
\begin{proof}
	The existence of $\mathcal{F}$ follows from Lemma \ref{lem:weierstrass}. Equation \eqref{eq:ref} shows that $\Re(\mathcal{F})$ is A-minimal. Furthermore, $\Re(i\mathcal{F})$ can be written in the form of Equation \eqref{eq:imf}. Exploiting Theorem \ref{thm:conjugate}, we conclude $\Re(i\mathcal{F})$ is C-minimal.
	
	The converse is straightforward. Given a conjugate pair of discrete minimal surfaces $(f,\tilde{f})$ with Gauss map $N$, we define $q:E_{int} \to \field{R}$ such that
	\[
	df(e^*_{ij}) = \frac{q_{ij}(1+|z_i|^2) (1+|z_j|^2)}{2|z_j-z_i|^2}  (N_j - N_i)
	\]
	where $z$ is the stereographic projection of $N$. Then it can be shown directly that $q$ is a holomorphic quadratic differential.
\end{proof}

\begin{corollary}
	The associated family of discrete surfaces $f^{\theta}$ in Definition \ref{def:assoicatedfamily} satisfies
	\[
	f^{\theta}= \Re(e^{i\theta} \mathcal{F}).
	\]
\end{corollary}

\subsection{Goursat transformations}
In contrast to Bonnet transformations in the smooth theory, Goursat transformations generate non-isometric minimal surfaces in general \cite{Goursat1887}. A conjugate pair of minimal surfaces can be regarded as a holomorphic null curve in $\field{C}^3$. A Goursat transform of a conjugate pair of minimal surfaces is a complex rotation acting on the null curve. These transformations preserve respectively the curvature line parametrizations and the asymptotic line parametrizations of minimal surfaces \cite{Mladenov2000}.

The M\"{o}bius invariant property of discrete holomorphic quadratic differentials (Proposition \ref{thm:mobius}) implies that if $q:E_{int} \to \field{R}$ is a discrete holomorphic quadratic differential on a non-degenerate realization $z:M \to \field{C}$, then the dual 1-form defined by
\begin{align} \label{eq:etagoursat}
	\eta_{\Phi}(e^*_{ij}) := \frac{q_{ij}}{(\Phi(z_j) -\Phi(z_i))} \left( \begin{array}{c} 1-\Phi(z_i)\Phi(z_j) \\ i (1+\Phi(z_i)\Phi(z_j)) \\ \Phi(z_i)+\Phi(z_j) \end{array} \right)
\end{align}
is closed for any M\"{o}bius transformation $\Phi:\field{C}\cup \{\infty\} \to \field{C}\cup \{\infty\}$. Note every M\"{o}bius transformation $\Phi$ can be represented as
\begin{align}
	\Phi(z) = \frac{az+b}{cz+d}  \label{eq:mobiustransform}
\end{align}
for some $a,b,c,d \in \field{C}$ with $ad-bc=1$. We are going to see how a minimal surface deforms if its Gauss map is transformed under a M\"{o}bius transformation by substituting \eqref{eq:mobiustransform} into \eqref{eq:etagoursat}.

The following can be verified directly.
\begin{lemma} \label{lem:goursat}
	Let $\Phi(z):= \frac{az+b}{cz+d}$ with $ad-bc=1$. Then
	\begin{align*}
		\frac{1}{\Phi(z_j) -\Phi(z_i)} \left( \begin{array}{c} 1-\Phi(z_i)\Phi(z_j) \\ i (1+\Phi(z_i)\Phi(z_j)) \\ \Phi(z_i)+\Phi(z_j) \end{array} \right) =  \frac{1}{z_j -z_i}A_{\Phi} \left( \begin{array}{c}
			1-z_i z_j \\ i(1+z_i z_j) \\ z_i + z_j
		\end{array}\right)
	\end{align*}
	where
	\[
	A_{\Phi}:=\left( \begin{array}{ccc} \frac{1}{2}(a^2-b^2-c^2+d^2) & \frac{i}{2}(a^2+b^2-c^2-d^2) & -ab+cd \\ \frac{i}{2}(-a^2+b^2-c^2+d^2) & \frac{1}{2}(a^2+b^2+c^2+d^2) & i(ab+cd) \\ -ac+bd & -i(ac+bd) & ad+bc \end{array} \right).
	\]
\end{lemma} 
The following indicates that conjugate pairs of discrete minimal surfaces deform exactly the same way as the smooth ones under Goursat transforms.
\begin{theorem}
	Given a discrete surface $M$ and its dual $M^*$, we suppose $f,\tilde{f}:V^* \to \field{R}^3$ form a conjugate pair of minimal surfaces with a non-degenerate admissible Gauss map $N:V \to \field{S}^2$ where $f$ is A-minimal and $\tilde{f}$ is C-minimal. For any M\"{o}bius transformation
	\[
	\Phi(z):=\frac{az+b}{cz+d}, \quad \text{where } a,b,c,d \in \field{C} \text{ and } ad-bc=1
	\]
	such that $N_{\Phi}:V \to \field{S}^2$ defined by
	\[
	N_{\Phi} := \frac{1}{|\Phi(z)|^2+1}\left(\begin{array}{c}2\Re \Phi(z) \\ 2\Im \Phi(z) \\ |\Phi(z)|^2-1\end{array}\right) \quad \text{where } z:= \frac{N_1}{1-N_3} + i \frac{N_2}{1-N_3},
	\]
	is admissible, we consider
	\[
	A_{\Phi}:=\left( \begin{array}{ccc} \frac{1}{2}(a^2-b^2-c^2+d^2) & \frac{i}{2}(a^2+b^2-c^2-d^2) & -ab+cd \\ \frac{i}{2}(-a^2+b^2-c^2+d^2) & \frac{1}{2}(a^2+b^2+c^2+d^2) & i(ab+cd) \\ -ac+bd & -i(ac+bd) & ad+bc \end{array} \right)
	\]
	and define $f_{\Phi},\tilde{f}_{\Phi}:V^* \to \field{R}^3$ by
	\[
	f_{\Phi}- i\tilde{f}_{\Phi}:= A_{\Phi} (f- i\tilde{f}).
	\]
	Then, $f_{\Phi}$ is A-minimal and $\tilde{f}_{\Phi}$ is C-minimal. They form a conjugate pair of minimal surfaces with Gauss map $N_{\Phi}$.  
\end{theorem}
\begin{proof}
	Note $f_{\Phi},\tilde{f}_{\Phi}$ are well defined on $M^*$ even if $M^*$ is not simply connected. We denote by $z:V \to \field{C}$ the stereographic projection of $N$. From the proof the Weierstrass representation (Theorem \ref{thm:Weierstrass}), there exists a holomorphic quadratic differential $q:E_{int} \to \field{R}$ such that
	\[
	df(e^*_{ij}) - id\tilde{f}(e^*_{ij}) =\frac{q_{ij}}{dz(e_{ij})} \left( \begin{array}{c}
	1-z_i z_j \\ i(1+z_i z_j) \\ z_i + z_j
	\end{array}\right).
	\]
	Then, by lemma \ref{lem:goursat}
	\begin{align*}
		df_{\Phi}(e^*_{ij}) - id\tilde{f}_{\Phi}(e^*_{ij}) =&\frac{q_{ij} }{dz(e_{ij})} A_{\Phi} \left( \begin{array}{c}
			1-z_i z_j \\ i(1+z_i z_j) \\ z_i + z_j
		\end{array}\right) \\
		=& \frac{q_{ij}}{(\Phi(z_j) -\Phi(z_i))} \left( \begin{array}{c} 1-\Phi(z_i)\Phi(z_j) \\ i (1+\Phi(z_i)\Phi(z_j)) \\ \Phi(z_i)+\Phi(z_j) \end{array} \right).
	\end{align*}
	The proof of the Weierstrass representation (Theorem \ref{thm:Weierstrass}) shows that $f_{\Phi}$ is A-minimal and $\tilde{f}_{\Phi}$ is C-minimal.
\end{proof}
\begin{figure}[h]
	\centering
	\includegraphics[width=0.9\textwidth]{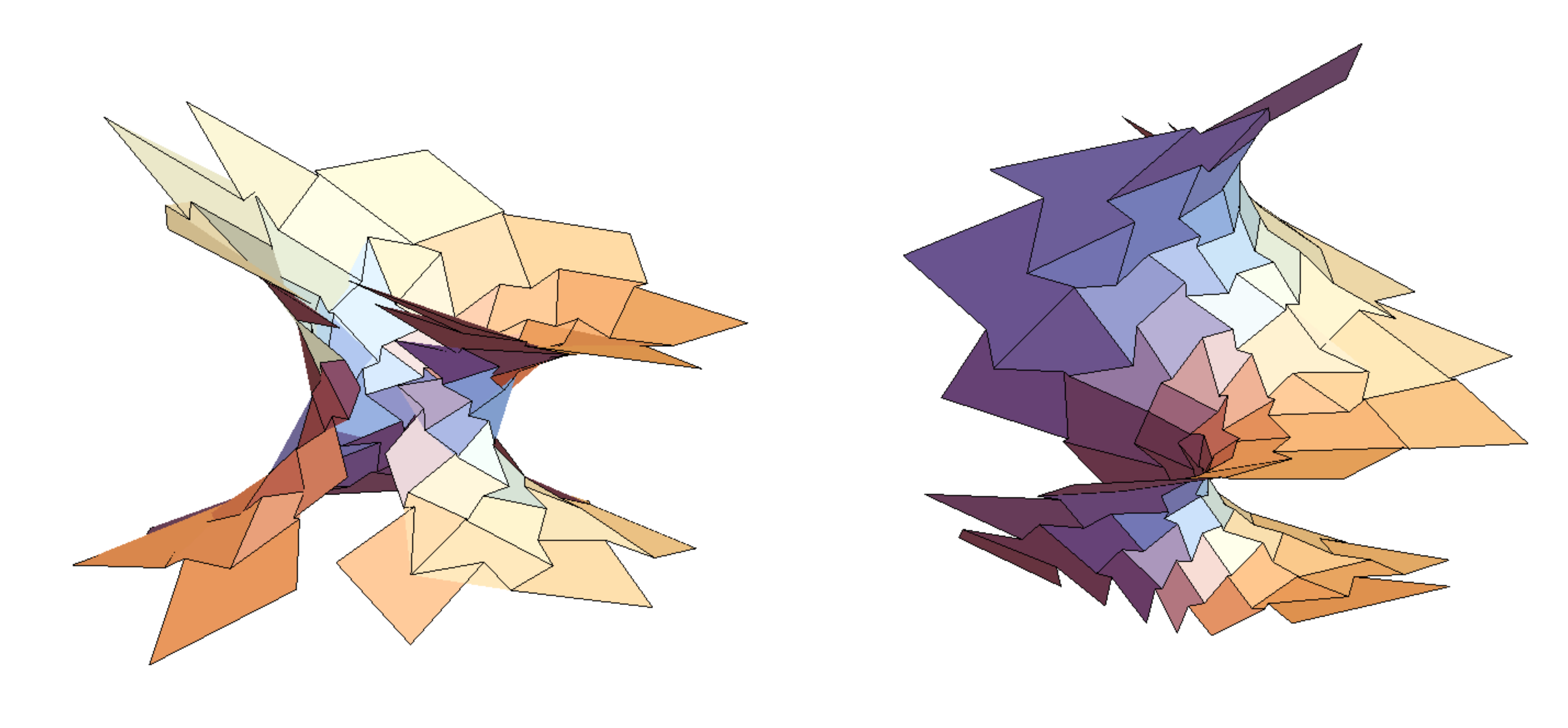}
	\caption{A conjugate pair of discrete minimal surfaces obtained via a Goursat transform (a complex rotation) of the conjugate pair in Figure \ref{fig:conjugate}.}
	\label{fig:goursat}
\end{figure}

\section{Self-stresses}\label{sec:selfstress}

The Weierstrass representation asserts that a smooth minimal surface is locally determined by its Gauss map together with a holomorphic quadratic differential. A holomorphic quadratic differential in this case can be interpreted as a static stress \cite{Smyth2004}. 

In this section, we provide a discrete version of the above statement: each discrete minimal surface corresponds to a \emph{self-stress} on its Gauss map. It is closely related to Maxwell theorem asserting that a self-stress, i.e. an assignment of forces along the edges of a realization balanced at vertices, corresponds to a \emph{reciprocal parallel mesh} of the realization \cite{Maxwell1870}. Here we focus on triangulated surfaces for simplicity.

There are several results concerning self-stresses and discrete minimal surfaces \cite{Wallner2008,Muller2010,Schief2014}, though different from the consideration here.

The following is immediate from the definition of A-minimal surfaces.
\begin{proposition} \label{thm:selfstress}
	Suppose we have a non-degenerate admissible realization $N:V \to \field{S}^2$ of a simply connected triangulated surface $M$. Given a function $k:E_{int} \to \field{R}$, the following are equivalent: 
	\begin{enumerate}
		\item There exists an A-minimal surface $f:V^* \to \field{R}^3$ with Gauss map $N$ satisfying for every interior edge $\{ij\}$
		\[
		df(e_{ij}^*)=k_{ij} (N_j- N_i).
		\]
		\item The assignment of $k_{ij} (N_j-N_i)$ to each oriented edge $e_{ij}$ defines equal and opposite forces to the two endpoints along every edge that is in equilibrium at every vertex, i.e. for every interior vertex $i$
		\[
		\sum_j k_{ij} (N_j-N_i) =0.
		\]
	\end{enumerate}

\end{proposition}

There is a similar correspondence between a C-minimal surface and the polar of its Gauss map. Given a point $\hat{x} \in \field{R}^3$ its \emph{polar plane} with respect to the unit sphere is defined as
\[
\bar{x}:=\{y \in \field{R}^3| \langle y, \hat{x} \rangle=1\}
\]
and $\hat{x}$ is called the \emph{pole} of $\bar{x}$. If $N:V \to \field{S}^2$ is a non-degenerate admissible realization of a triangulated surface in such a way that neighboring faces are not coplanar, then the pole of each face determines a non-degenerate realization $\hat{N}:V^* \to \field{R}^3$ of $M^*$ with planar faces tangent to the unit sphere and with face normal $N$. In particular the image of each dual edge $\{ij\}^*$ under $\hat{N}$ is parallel to $N_i \times N_j$.

\begin{lemma} \label{lem:polar}
	Let $N,N_1\dots N_r \in \field{S}^2$ and $k_1,k_2 \dots k_r \in \field{R}$. Then
	\[
	\sum_{j=1}^r k_j (N_j - N) =0 \quad \iff  \begin{cases}
	\sum_{j=1}^r k_j N \times N_j=0 \\
	\sum_{j=1}^r k_j r_j \times (N \times N_j)=0
	\end{cases}
	\]
	where $r_j := (N + N_j)/(1+\langle N,N_j \rangle)$ is a point on the line $\{x \in \field{R}^3| \langle x, N \rangle=1 \} \cap \{x \in \field{R}^3| \langle x, N_j \rangle=1 \}$. 
\end{lemma}
\begin{proof}
	It follows from the identities that
	\[
	\sum_{j=1}^r k_{j} N \times N_j = N \times \sum_{j=1}^r k_{j} (N_j - N)
	\]
	and
	\begin{align*}
		\sum_{j=1}^r k_j r_j \times (N \times N_j) &= \sum_{j=1}^r k_j \frac{-(N_j - \langle N,N_j \rangle N) + (N - \langle N,N_j \rangle N_j)}{(1+\langle N,N_j \rangle)} \\
		&= \sum_{j=1}^r k_j (N - N_j). 
	\end{align*}
\end{proof}

Combining Theorem \ref{thm:conjugate}, Theorem \ref{thm:selfstress} and Lemma \ref{lem:polar}, we have
\begin{proposition}
	Suppose a realization $N:V \to \field{S}^2$ of a simply connected triangulated surface $M$ and its polar $\hat{N}:V^* \to \field{R}^3$ with respect to the unit sphere are non-degenerate. Given $k:E_{int} \to \field{R}$, the following are equivalent: 
	\begin{enumerate}
		\item There exists a C-minimal surface $\tilde{f}:V^* \to \field{R}^3$ with Gauss map $N$ such that for every interior edge $\{ij\}$
		\[
		df(e_{ij}^*)=k_{ij} N_i \times N_j.
		\]
		\item The assignment of $k_{ij} N_i \times N_j$ to each oriented edge $e^*_{ij}$ defines equal and opposite forces to the neighboring faces along each edge of $\hat{N}$ that is in equilibrium on every face, i.e. for every dual face $i^* \in F^*_{int}$
		\begin{align*}
			\sum_j k_{ij} N_i \times N_j &=0 \quad \quad \text{(forces balanced)}, \\
			\sum_j k_{ij} r_{ij} \times (N_i \times N_j) &=0 \quad \quad \text{(torques balanced)}
		\end{align*}
		where $r_{ij} := (N + N_j)/(1+\langle N,N_j \rangle)$ is a point on the image of $\{ij\}^*$ under $\hat{N}$.
	\end{enumerate} 
\end{proposition}

\begin{remark}
	The discussion above can be applied to discrete minimal surfaces with general combinatorics. By vertex splitting, every discrete minimal surface can be regarded as trivalent with its Gauss map triangulated. Here discrete minimal surfaces are allowed to have degenerate edges. 
\end{remark}

\section{Critical points of the total area}\label{sec:criticalarea}

In this section, we show that among the discrete minimal surfaces that we have discussed, there is a subclass of them which possess the variational property of minimal surfaces. 

\begin{figure}[h]
	\centering
	\includegraphics[width=0.5\textwidth]{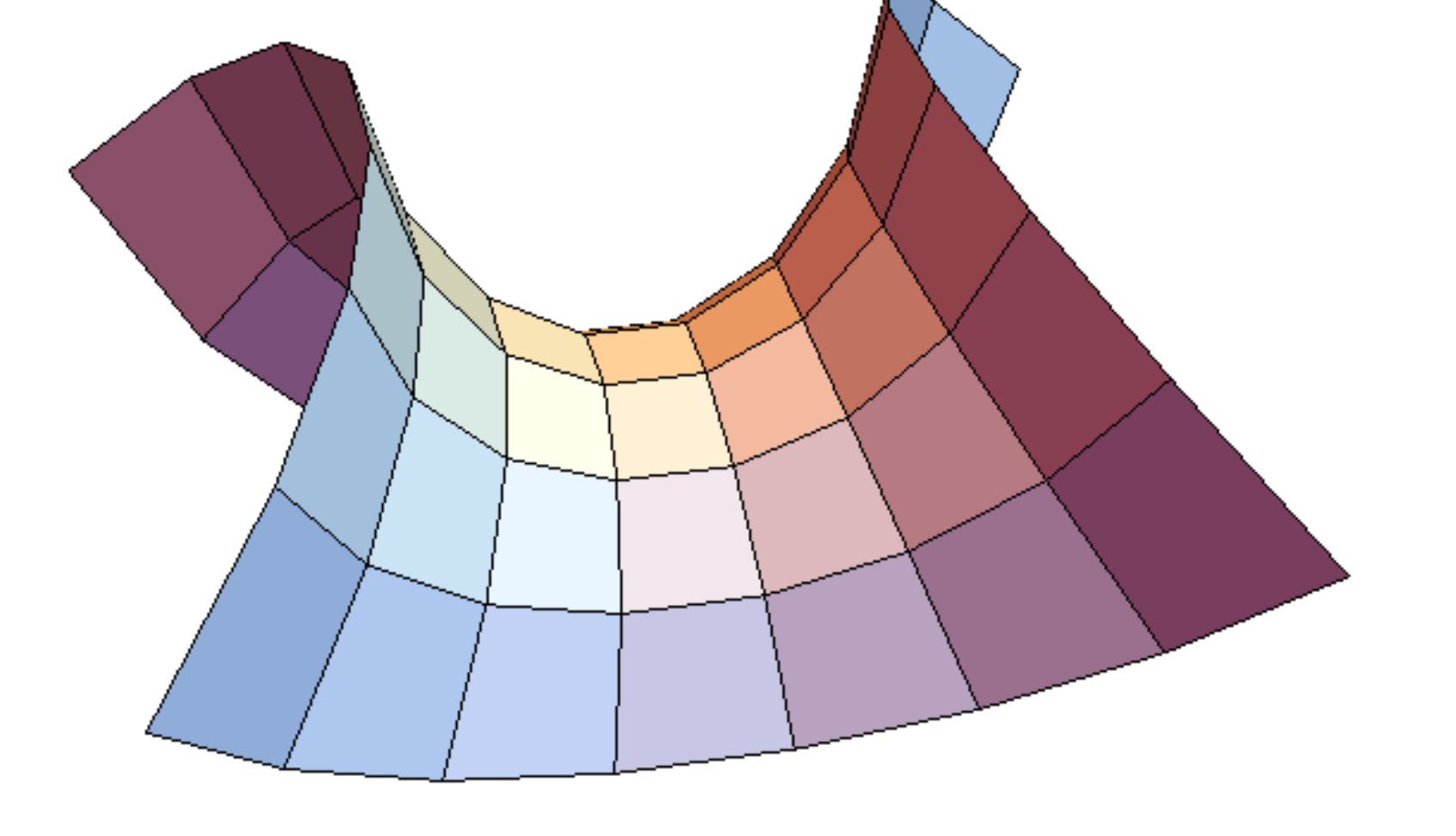}
	\caption{A C-minimal discrete surface which is a critical point of the area functional. Subdivision by adding diagonals yields a triangulated minimal surface in the sense of Pinkall and Polthier \cite{Pinkall1993}.}
	\label{fig:ennerper}
\end{figure}

\subsection{Area of non-planar faces}
Discrete surfaces in the associated families do not have planar faces in general. In order to define area on a non-planar face, we consider its vector area. This idea has been applied in computer graphics \cite{Alexa2011}.

Given a polygon $\gamma=(\gamma_0,\gamma_1,\dots,\gamma_n=\gamma_0)$ in $\field{R}^3$, its \emph{vector area} is defined by
\[
\vec{A}_\gamma = \frac{1}{2}\sum_{i=0}^{n-1} \gamma_i \times \gamma_{i+1}.
\]

The vector area is invariant under translations of the polygon. The magnitude $|\vec{A}_\gamma|$ is the largest signed area over all orthogonal projections of $\gamma$ to planes in space. The direction of $\vec{A}_\gamma$ indicates the normal of the plane with largest signed area. If $\gamma$ is embedded in a plane, the magnitude $|\vec{A}_\gamma|$ coincides with the usual notion of area and the direction of $\vec{A}$ is normal to the plane.

However, there is still an ambiguity to define the (signed) area $A_\gamma$ of a non-planar polygon $\gamma$ using the vector area, either $A_\gamma:=|\vec{A}_\gamma|$ or $-|\vec{A}_\gamma|$. Such ambiguity will be fixed using the Gauss map of a discrete minimal surface in Section \ref{sec:minifromp}. Assuming we have picked a sign for the area of each face, we derive the gradient of the total area of a discrete surface.

\begin{definition}
	Suppose $f:V \to \field{R}^3$ is a realization of a compact discrete surface $M=(V,E,F)$ with boundary. Let $\sigma: F \to \pm 1$ be a choice of signs. We have the vector area $\vec{A}:F \to \field{R}^3$ given by 
	\[
	\vec{A}_{\phi} = \frac{1}{2}\sum_{i=0}^{n-1} f_i \times f_{i+1}
	\]
	where $(v_0,v_1,\dots,v_{n-1})= \phi \in F$. We define the \emph{total area} $\area_{\sigma}(f)$ by
	\[
	\area_{\sigma}(f) = \sum_{\phi \in F} \sigma_{\phi} |\vec{A}_{\phi}|.
	\]
	and, if $\vec{A}\neq 0 $, the \emph{mean curvature vector field} $\vec{H}_{\sigma}:V_{int} \to \field{R}^3$ by
	\[
	\vec{H}_{\sigma,i} := \frac{1}{2}\sum_{j} dN_{\sigma}(e^*_{ij}) \times df(e_{ij}) \quad \forall \, i \in V_{int}
	\]
	where $N_{\sigma}:F \to \field{S}^2$ is given by $N_{\sigma} := \sigma \vec{A}/|\vec{A}|$.
\end{definition}

If $f_t$ is a family of realizations with $f_0=f$ and $\vec{A}_t$ is nonzero on any face, then the total area $\area_{\sigma}(f_t)$ depends smoothly on $t$.

\begin{proposition}\label{thm:gradarea}
	Suppose $f:V \to \field{R}^3$ is a realization of a compact discrete surface $M=(V,E,F)$ with boundary and with non-vanishing vector area $\vec{A}:F \to \field{R}^3\backslash \{0\}$. Let $\sigma: F \to \pm 1$ be a choice of signs. Then the mean curvature vector field $\vec{H}_{\sigma}$ vanishes identically if and only if $f$ is a critical point of the total area $\area_{\sigma}(f)$ under infinitesimal deformations with the boundary fixed.
	
\end{proposition}
\begin{proof}
	Suppose $\dot{f}:V \to \field{R}^3$ is an infinitesimal deformation of $f$ fixing the boundary. We define $N_{\sigma}:F \to \field{S}^2$ by $N_{\sigma}:= \sigma \vec{A}/|\vec{A}|$. We consider a face $(v_0,v_1,\dots,v_{n-1})= \phi \in F$ and write $f_i= f(v_i)$.
	
	Since $ \langle N_{\sigma,\phi}, \dot{N}_{\sigma,\phi} \rangle =0$ we have
	\[
	\langle \vec{A}_\phi, \dot{N}_{\sigma,\phi} \rangle =0.
	\]
	On the other hand,	
	\begin{align*}
		\langle \dot{\vec{A}}_\phi, N_{\sigma,\phi} \rangle
		&= \frac{1}{2}\sum_{i=0}^{n-1} \langle f_{i-1} \times \dot{f}_{i} + \dot{f}_i \times f_{i+1}, N_{\sigma,\phi} \rangle \\
		&= \frac{1}{2}\sum_{i=0}^{n-1} \langle (f_{i+1} - f_i + f_i -f_{i-1}) \times N_{\sigma,\phi}, \dot{f}_{i} \rangle \\
		&= \frac{1}{2}\sum_{i=0}^{n-1} \langle  (f_{i+1}-f_{i}) \times N_{\sigma,\phi}+ (f_{i-1}-f_{i}) \times (-N_{\sigma,\phi}), \dot{f}_{i} \rangle.
	\end{align*}
	Thus,
	\begin{align*}
		\dot{\area_{\sigma}(f)} =& \sum_{\phi \in F} \langle \dot{\vec{A}}_\phi, N_{\sigma,\phi} \rangle + \langle \vec{A}_\phi, \dot{N}_{\sigma,\phi} \rangle \\
		=& \frac{1}{2}\sum_{i\in V_{int}} \langle \sum_{j} dN_{\sigma}(e^*_{ij}) \times df(e_{ij}), \dot{f}_i \rangle \\
		=& \sum_{i\in V_{int}} \langle \vec{H}_{\sigma,i}, \dot{f}_i \rangle.
	\end{align*}
	It implies $f$ is a critical point of the area functional $\area_{\sigma}(f)$ under infinitesimal deformations with the boundary fixed if and only if its mean curvature vector field $\vec{H}_{\sigma}$ vanishes identically.
\end{proof}

\begin{figure}[h]
	\centering
	\includegraphics[width=0.8\textwidth]{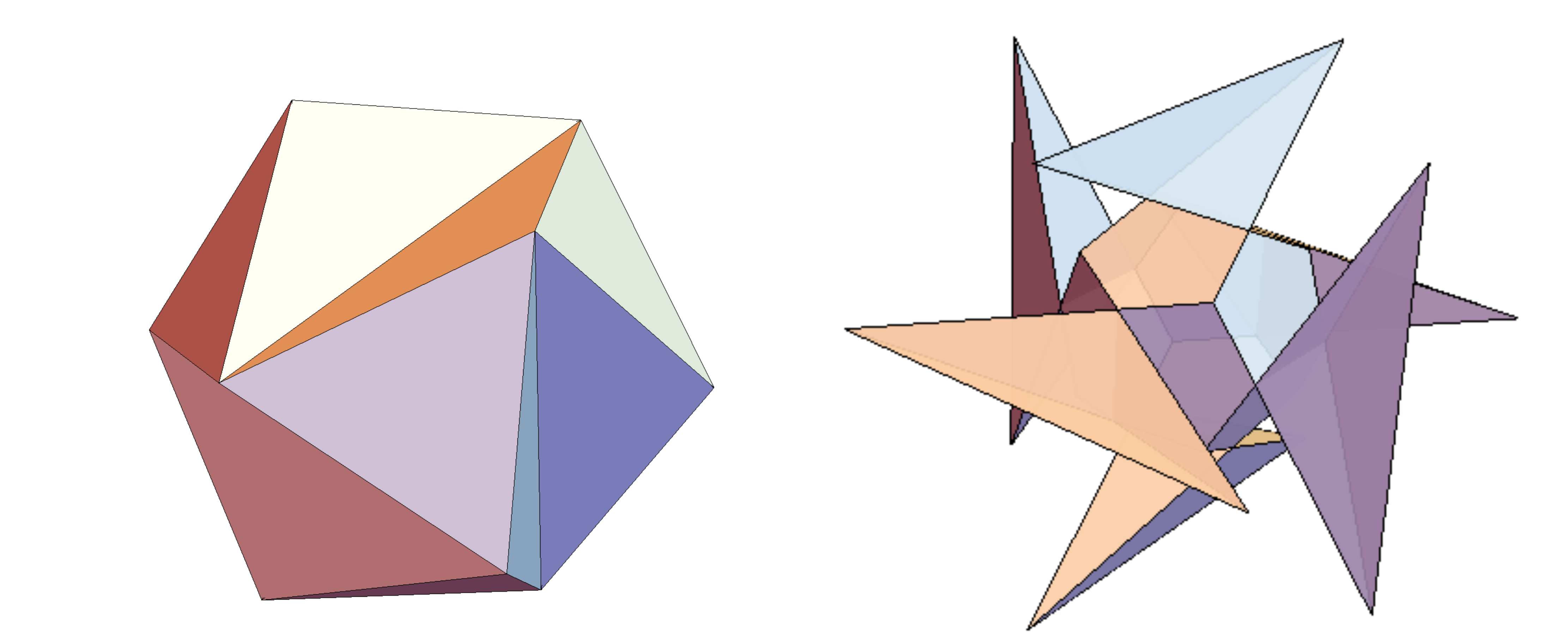}
	\caption{Jessen's orthogonal icosahedron (left) is known to be infinitesimally flexible. Its infinitesimal isometric deformation yields a \emph{closed} C-minimal surface (right) with non-embedded planar faces. For each face of the discrete minimal
		surface, the vector area vanishes.}
	\label{fig:jessen}
\end{figure}

Before ending this section, we relate mean curvature vector field with the cotangent formula by Pinkall and Polthier \cite{Pinkall1993}. 

\begin{corollary}\label{cor:cotangentminimal}
	Suppose $f:V \to \field{R}^3$ is a realization of a triangulated surface such that each face spans an affine plane. Then
	\begin{equation*}
		\sum_{j} dN(e^*_{ij}) \times df(e_{ij}) = \sum_{j} (\cot \angle jki + \cot \angle ilj)  df(e_{ij})   \label{eq:areaminimalg}
	\end{equation*}
	where $\{ijk\}$ and $\{jli\}$ are two neighboring faces containing the edge $\{ij\}$ and $N:= \vec{A}/|\vec{A}|$ is the face normal field given by the orientation of the triangulated surface. 
	
	Hence, a realization $f:V \to \field{R}^3$ of a compact triangulated surface is a critical point of the area functional $\Area_{\sigma}(f)$, where $\sigma\equiv 1$, under infinitesimal deformations with the boundary fixed if and only if for every interior vertex $i$
	\[
	\sum_j (\cot \angle jki + \cot \angle ilj) (f_j - f_i) =0.
	\]
\end{corollary}
\begin{proof}It follows from Proposition \ref{thm:gradarea} and the following identity that for every interior vertex $i$
	\begin{align*}
		\sum_{j} dN(e^*_{ij}) \times df(e_{ij})  
		=& \sum_{ijk} N_{ijk} \times (df(e_{ij}) - df(e_{ik})) \\
		=& \sum_{ijk} N_{ijk} \times df(e_{kj}) \\
		=& \sum_{ijk} \cot \angle jki \, df(e_{ij}) - \cot \angle ijk \, df(e_{ki}) \\
		=&\sum_{j} (\cot \angle jki + \cot \angle ilj)  df(e_{ij}).
	\end{align*}
\end{proof}

\subsection{Minimal surfaces from P-nets} \label{sec:minifromp}

We consider a P-net $z:V \to \field{C}$, which is a quadrivalent discrete surface possessing integrable structures as introduced in Definition \ref{def:pnet}. There is a canonical discrete holomorphic quadratic differential given by its P-labeling $\mu:E \to \pm 1$. Via the Weierstrass representation (Theorem \ref{thm:Weierstrass}), the holomorphic quadratic differential $\mu$ yields a conjugate pair $(f, \tilde{f})$ of discrete minimal surfaces. Denoting $N:V \to \field{S}^2$ the inverse stereographic projection of $z$
\begin{equation*}
	N := \frac{1}{1+|z|^2} \left(\begin{array}{c}2\Re z \\ 2\Im z \\ |z|^2-1\end{array}\right),		
\end{equation*}
then $f$ is an A-minimal surface satisfying 
\begin{align} \label{eq:aminimal}
	df(e^*_{ij}) & =\mu_{ij}\frac{N_j-N_i}{|N_j-N_i|^2}
\end{align}
and $\tilde{f}$ is a C-minimal surface satisfying
\begin{align} \label{eq:cminimal}
	d\tilde{f}(e^*_{ij}) = N_i \times \mu_{ij} \frac{N_j-N_i}{|N_j-N_i|^2} =N_j \times \mu_{ij} \frac{N_j-N_i}{|N_j-N_i|^2}.
\end{align}

We first consider the vector area of the discrete surfaces in the associated family. The following illustrates a discrete counterpart of the property that the area 2-form and the Gauss map of a smooth minimal surface is unchanged within the associated family. 

\begin{proposition}\label{lem:vecareaconst}
	Suppose $D$ is a P-graph and $z: V \to \field{C}$ is a P-net. We denote by $N:V \to \field{S}^2$ the stereographic projection of $z$ and write $f,\tilde{f}:V^* \to \field{R}^3$ as the corresponding conjugate minimal surfaces given by \eqref{eq:aminimal} and \eqref{eq:cminimal}. Considering the vector area $\vec{A}^{\theta}: F^*_{int} \to \field{R}^3$ of a discrete surface $f^{\theta}:= (\cos \theta) f + (\sin \theta) \tilde{f}$ in the associated family, then we have $\vec{A}^{\theta}$ independent of $\theta$ and parallel to $N$. 
\end{proposition}

\begin{proof}
	We focus on a dual face $v^*_0 \in F^*_{int}$, which corresponds to $v_0 \in V_{int}$ (see Figure \ref{fig:indices}), and decompose the neighboring edges into components
	\[
	df(e^*_i) = df(e^*_i)^{\perp} + df(e^*_i)^{\parallel} \quad  \forall \, i=1,2,3,4
	\]
	such that
	\[
	df(e^*_i)^{\perp} \perp N_0, \quad df(e^*_i)^{\parallel} \parallel N_0.
	\]
	In particular, since Equation \eqref{eq:aminimal} implies
	\[ \langle df(e^*_i), N_0 \rangle = \mu_{i} \frac{\langle N_i,N_0 \rangle-1}{|N_i -N_0|^2}= -\frac{\mu_{i}}{2}, \]
	we have 
	\[
	df(e^*_i)^{\parallel} = -\mu_{i} \frac{N_0}{2}
	\]
	
	On the other hand, the C-minimal surface $\tilde{f}$ satisfies
	\begin{align*}
		d\tilde{f}(e^*_i) &= N_0 \times df(e^*_i) = N_0 \times df(e^*_i)^{\perp}. 
	\end{align*}
	It yields
	\begin{align*}
		df^{\theta}(e^*_i) &= df^{\theta}(e^*_i)^{\perp} + df^{\theta}(e^*_i)^{\parallel} = R_{\theta}( df(e^*_i)^{\perp} ) - \mu_i \frac{\cos \theta}{2} N_0
	\end{align*}
	where $R_{\theta}(v):= \cos \theta \,  v^{\perp} + \sin \theta \, N_0 \times v^{\perp}$ for any $v^{\perp} \perp N_0$ is a rotation in the plane $N_0^\perp$. We calculate the vector area on the dual face $v^*_0$
	\begin{align*}
		2\vec{A}^{\theta}_0 =& df^{\theta}(e^*_1) \times df^{\theta}(e^*_2) + df^{\theta}(e^*_3) \times df^{\theta}(e^*_4) \\
		=& df^{\theta}(e^*_1)^{\perp} \times df^{\theta}(e^*_2)^{\perp} + df^{\theta}(e^*_3)^{\perp} \times df^{\theta}(e^*_4)^{\perp} \\
		&- \frac{\cos \theta}{2} N_0 \times (\mu_1 df^{\theta}(e^*_2)^{\perp} - \mu_2 df^{\theta}(e^*_1)^{\perp} + \mu_3 df^{\theta}(e^*_4)^{\perp} - \mu_4 df^{\theta}(e^*_3)^{\perp}) \\
		=& df^{\theta}(e^*_1)^{\perp} \times df^{\theta}(e^*_2)^{\perp} + df^{\theta}(e^*_3)^{\perp} \times df^{\theta}(e^*_4)^{\perp} \\
		&- \mu_1 \frac{\cos \theta}{2} N_0 \times (df^{\theta}(e^*_2)+  df^{\theta}(e^*_1)+ df^{\theta}(e^*_4)+ df^{\theta}(e^*_3))^{\perp} \\
		=&	df^{\theta}(e^*_1)^{\perp} \times df^{\theta}(e^*_2)^{\perp} + df^{\theta}(e^*_3)^{\perp} \times df^{\theta}(e^*_4)^{\perp} \\
		=& R_{\theta}(df(e^*_1)^{\perp} \times df(e^*_2)^{\perp} + df(e^*_3)^{\perp} \times df(e^*_4)^{\perp})\\
		=& df(e^*_1)^{\perp} \times df(e^*_2)^{\perp} + df(e^*_3)^{\perp} \times df(e^*_4)^{\perp}			
	\end{align*}
	which is independent of $\theta$ and parallel to $N_0$. Here we made use of the property of $\mu$:
	\[
	\mu_1 = -\mu_2 = \mu_3 = -\mu_4.
	\]
\end{proof}

Now we show that discrete surfaces in the associated family from P-nets possess vanishing mean curvature vector fields.

\begin{lemma} \label{lem:edgecrossnormal}
	\begin{align*}
		df^{\theta}(e^*_{ij}) \times dN(e_{ij}) &= -\mu_{ij} \sin \theta \,(N_i + N_j)/2  \\
		\langle df^{\theta}(e^*_{ij}),dN(e_{ij}) \rangle &= \mu_{ij} \cos \theta
	\end{align*}
\end{lemma}
\begin{proof}
	We have by definition
	\begin{align*}
		df(e^*_{ij}) \times dN(e_{ij}) &= 0  \\
	\end{align*}
	and
	\begin{align*}
		& d\tilde{f}(e^*_{ij}) \times  dN(e_{ij}) \\
		=& \mu_{ij} \big((N_j \times \frac{N_j-N_i}{|N_j-N_i|^2}) \times N_j - (N_i \times \frac{N_j-N_i}{|N_j-N_i|^2}) \times N_i \big)   \\
		=&-\mu_{ij} \big(\frac{N_i- \langle N_i,N_j \rangle N_j}{|N_j-N_i|^2} +  \frac{N_j-\langle N_i, N_j \rangle N_i}{|N_j-N_i|^2}\big)  \\
		=& \mu_{ij} \frac{\langle N_i, N_j \rangle -1}{|N_j-N_i|^2} (N_i+N_j)\\
		=& -\mu_{ij}(N_i+N_j)/2.
	\end{align*}
	Since  $f^\theta = \cos \theta \, f + \sin \theta \, \tilde{f}$ we get
	\[
	df^{\theta}(e^*_{ij}) \times dN(e_{ij}) = -\mu_{ij} \sin \theta \,(N_i + N_j)/2.
	\]
	On the other hand we have
	\begin{align*}
		\langle df(e^*_{ij}),dN(e_{ij}) \rangle &= \mu_{ij}, \\
		\langle d\tilde{f}(e^*_{ij}),dN(e_{ij}) \rangle&= 0.
	\end{align*}
	Hence
	\[
	\langle df^{\theta}(e^*_{ij}),dN(e_{ij}) \rangle = \mu_{ij} \cos \theta.
	\]
\end{proof}

Suppose $D$ is a P-graph with boundary. We call the faces of $D$ containing boundary edges as boundary faces. Those non-boundary faces are called interior, which form a set $F_{int}$. We denote the set of dual vertices which correspond to $F_{int}$ by $V_{int}^*$. Applying Proposition \ref{lem:vecareaconst} and Lemma \ref{lem:edgecrossnormal}, we have the following.
\begin{theorem}
	Given a P-graph $D=(V,E,F)$ and a P-net $z:V \to \field{R}^2 \cong \field{C}$, we let $f^{\theta}:V^* \to \field{R}^3$ be a discrete surface in the corresponding associated family with non-vanishing vector area $\vec{A}:F^*_{int} \to \field{R}^3 \backslash \{0\}$. We denote the inverse stereographic projection of $z$ by $N$ and define $\sigma:F^*_{int} \to \pm 1$ via
	\[
	\sigma:= \langle N, \vec{A}/|\vec{A}| \rangle.
	\]
	Then the mean curvature vector field $\vec{H}_{\sigma}^{\theta}:V^*_{int} \to \field{R}^3$ vanishes identically, i.e. for each dual vertex $\phi^*=(v_1,v_2,\dots,v_n)^*\in V^*_{int}$
	\begin{equation}\label{eq:vanmeancurvertex}
		\vec{H}^{\theta}_{\phi^*} := \frac{1}{2}\sum_{i} dN(e_{i,i+1}) \times  df^{\theta}(e^*_{i,i+1})=0
	\end{equation}
	and furthermore
	\begin{equation}\label{eq:vandotvertex}
		\sum_{i} \langle dN(e_{i,i+1}), df^{\theta}(e^*_{i,i+1}) \rangle =0.
	\end{equation}
	In particular if $D$ is compact, then $f^{\theta}$ is a critical point of the total area $\sum \sigma|\vec{A}|$ under infinitesimal deformations with boundary fixed.
\end{theorem}
\begin{proof}
	Proposition \ref{lem:vecareaconst} shows that the function $N_{\sigma}:F^*_{int} \to \field{S}^2$ satisfies
	\[
	N_{\sigma} := \sigma \vec{A}/|\vec{A}| = N.
	\]
	Then from Lemma \ref{lem:edgecrossnormal}, we have for each dual vertex $\phi^*=(v_1,v_2,\dots,v_n)^*\in V^*_{int}$
	\begin{align*}
		\vec{H}^{\theta}_{\phi^*} = \frac{1}{2}\sum_{i} dN(e_{i,i+1}) \times df^{\theta}(e^*_{i,i+1}) 
		=& \frac{\sin \theta}{4} \sum_i ( \mu_{i-1,i}+ \mu_{i,i+1}) N_i =0
	\end{align*}
	and
	\begin{align*}
		\sum_{i} \langle dN(e_{i,i+1}),  df^{\theta}(e^*_{i,i+1})\rangle =& \cos \theta \sum_{i} \mu_{i,i+1} = 0
	\end{align*}
	since the number of edges in a polygon $\phi$ is even.
\end{proof}

\begin{remark}
	The summations in \eqref{eq:vanmeancurvertex} and \eqref{eq:vandotvertex} are taken around a vertex of $f^{\theta}$ while those in \eqref{eq:vanmeancurfac} and \eqref{eq:vandotfac} are taken around a face. See Remark \ref{rmk:sumsmooth} for their analogue in the smooth theory.
\end{remark}

\begin{figure}[h]
	\centering
	\includegraphics[width=0.8\textwidth]{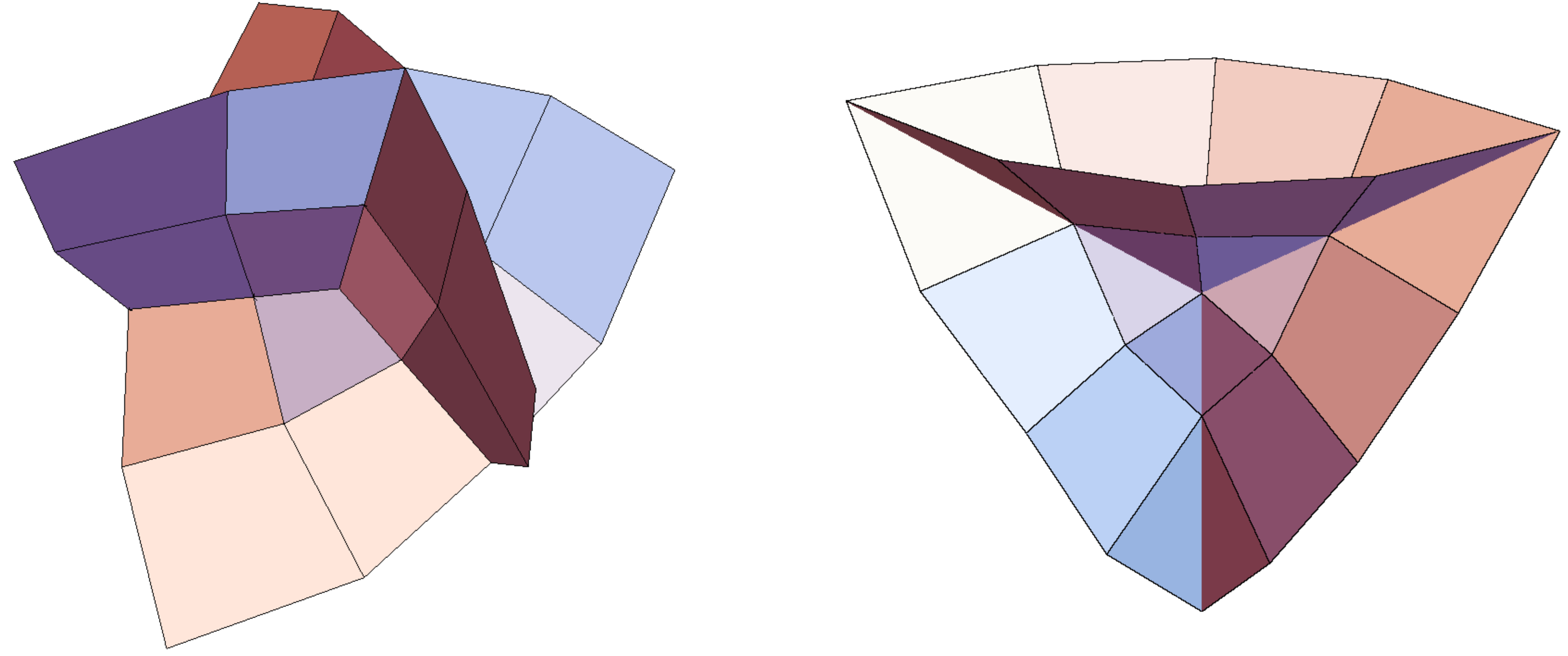}
	\caption{An A-minimal surface (left) and a C-minimal surface (right) obtained from the orthogonal circle pattern in Figure \ref{fig:planarcirc}. These discrete minimal surfaces are critical points of the total area.}
	\label{fig:circlepattern}
\end{figure}

\end{document}